\documentclass{article}

\usepackage{amsmath, amssymb, amsthm, amsfonts, color, slashed, xypic, mathrsfs, xy, amscd}

\newcommand{\definition}[2][]{\emph{#2}\index{#1#2}}

\newcommand{\id}{\mathrm{id}}
\newcommand{\Ad}{\mathrm{Ad}}
\newcommand{\Hom}{\mathrm{Hom}}
\newcommand{\End}{\mathrm{End}}
\newcommand{\ind}{\mathrm{Ind}}
\newcommand{\gr}{\mathrm{gr}}

\newcommand{\tr}{\mathrm{tr}}
\newcommand{\Aut}{\mathrm{Aut}}

\newcommand{\ad}{\mathrm{ad}}
\newcommand{\U}{\mathrm{U}}
\newcommand{\PU}{\mathrm{PU}}
\newcommand{\diag}{\mathrm{diag}}
\newcommand{\dR}{\mathrm{dR}}
\newcommand{\A}{\mathcal{A}}
\newcommand{\B}{\mathcal{B}}
\newcommand{\bb}{\flat}
\newcommand{\BB}{\textsc{b}}
\newcommand{\C}{\mathbb{C}}
\newcommand{\ch}{\mathrm{ch}}

\newcommand{\Chkr}{\mathrm{Chkr}}
\newcommand{\D}{\slashed{D}}

\newcommand{\Cl}{\mathbb{C}l}
\newcommand{\E}{\mathcal{E}}
\newcommand{\F}{\mathcal{F}}

\newcommand{\HH}{\bar{\mathcal
H}}

\newcommand{\K}{\mathrm{K}}
\newcommand{\sL}{\mathscr{L}}
\newcommand{\M}{\mathrm{M}}
\newcommand{\R}{\mathbb{R}}
\newcommand{\smooth}{C^{\infty}}
\newcommand{\Spin}{\mathrm{Spin}}
\newcommand{\SO}{\mathrm{SO}}
\newcommand{\Z}{\mathbb{Z}}
\newcommand{\ev}{\mathrm{ev}}
\newcommand{\odd}{\mathrm{odd}}

\newtheorem{Fact}{Fact}
\newtheorem{thm}{Theorem}[section]
\newtheorem{prop}[thm]{Proposition}

\newtheorem{cor}[thm]{Corollary}
\newtheorem{lem}[thm]{Lemma}
\theoremstyle{definition}
\newtheorem{defn}[thm]{Definition}

\theoremstyle{remark}
\newtheorem{remark}[thm]{Remark}

\begin{document}

\title{Projective Dirac operators, twisted K-theory, and local index formula \thanks{This article is the author's dissertation in publication form.}}
\author{
Dapeng Zhang \thanks{Supported in part by International Max-Planck Research School (IMPRS).}
\\California Institute of Technology, MC 253-37, Pasadena, CA 91125, USA
\\E-mail: zdapeng@caltech.edu
}
\date{}
\maketitle

\noindent \textbf{Abstract:}
We construct a canonical noncommutative spectral triple for every
oriented closed Riemannian manifold, which represents the fundamental class in the twisted K-homology of
the manifold. This so-called ``projective spectral triple'' is Morita
equivalent to the well-known commutative spin spectral triple
provided that the manifold is spin-c. We give an explicit local formula for the twisted Chern character for K-theories twisted with torsion classes, and with this formula we show that the twisted Chern character of the projective spectral triple is identical to the Poincar\'e dual of the A-hat genus of the manifold.

\medskip



\section*{Introduction}\addcontentsline{toc}{section}{Introduction}

The notion of spectral triple in Connes' noncommutative geometry
arises from extracting essential data from the K-homology part of
index theory in differential geometry. The following are basic
examples of commutative spectral triples:
\begin{enumerate}
\item The {\em spin spectral triple} for a spin$^c$ manifold $M$
with a spinor bundle $S$:
 $$\varsigma_1=(\smooth(M),\ \Gamma(S),\ \D,\ \omega),$$
where $\omega$ is the grading operator on $S$.

(Throughout this paper, unless otherwise stated explicitly, all vector spaces, algebras, differential forms, and vector bundles except cotangent bundles are considered over the field $\C$ of
{\em complex} numbers. For example, the notation $\smooth(M)$ is the same as $\smooth(M,\C)$. The notation $\Gamma(M,E)$ or $\Gamma(E)$ for a fibre bundle $E$ over $M$ always stands for the space of {\em smooth} sections of $E$.)

The identity between the analytic and topological indices of $\D$ is the
Atiyah-Singer index formula for the spin$^c$ manifold.
\item The spectral triple for the signature for a Riemannian
manifold $M$:
$$\varsigma_2=(\smooth(M),\ \Omega(M),\ d+d^*,\ *(-1)^{{\deg(\deg-1)\over 2}-{\dim M\over 4}}).$$
(For each function $f:\mathbb N\to \C$, we denote by $f(\deg):\Omega(M)\to \Omega(M)$  the linear operator given by
$$f(\deg)\,\omega = f(k)\,\omega, \  \forall \omega\in \Omega^k(M).)$$
The index formula corresponding to this spectral triple is the
Hirzebruch signature formula.
\item The spectral triple for the Euler characteristic for a Riemannian
manifold $M$:
$$\varsigma_3=(\smooth(M),\ \Omega(M),\ d+d^*,\ (-1)^{\deg}).$$

The local index formula corresponding to this spectral triple is the
Gauss-Bonnet-Chern theorem.
\end{enumerate}
In fact, every special case of Atiyah-Singer Index theorem
corresponds to an instance of commutative spectral triple (with
additional structures when necessary). These spectral triples, like
$\varsigma_1,\ \varsigma_2,\ \varsigma_3$, have many nice properties such as
``the five conditions'' in Connes \cite{reconstruction}, and
conversely, it is proved that \cite{reconstruction} any commutative
spectral triple $(\A,\mathcal H, D, \gamma)$ satisfying those five
conditions is equivalent to a spectral triple consisting of the
algebra of smooth functions on a Riemannian manifold $M$, the module
of sections of a Clifford bundle over $M$ and a Dirac type operator
on it. Furthermore, if $(\A,\mathcal H, D, \gamma)$ satisfies an
additional important property -- the Poincar\'e duality in K-theory
-- which means $(\A,\mathcal H, D, \gamma)$ represents the
fundamental class (i.e., a K-orientation) in $K^0(\A)$, then it is
equivalent to a spin spectral triple $\varsigma_1$ for some spin$^c$
manifold. The spectral triples $\varsigma_2$ and $\varsigma_3$ do not have the property of Poincar\'e
duality; however, we show in this paper (Corollary
\ref{projectivesptr}, Theorem \ref{duality}) that for every closed
oriented Riemannian manifold there is a canonical noncommutative
spectral triple having the property of
Poincar\'e duality in $K^0(M,W_3(M))$, the twisted K-theory of $M$ with local coefficient $W_3(M)$ -
 the third integral Stiefel-Whitney class. This canonical spectral triple is called the projective spectral triple on $M$, and its center is unitarily equivalent to $\varsigma_2$. The projective
spectral triple is Morita equivalent to the spin spectral triple
provided that the underlying manifold is spin$^c$. On the other hand, in
the paper of Mathai-Melrose-Singer \cite{Mathai}, a so-called
projective spin Dirac operator was defined for every Riemannian
manifold; however, this operator is in a formal sense. It turns out
that the projective spectral triple, in which the Dirac operator is
really an operator acting on a Hilbert space, just plays the role of
the projective spin Dirac operator.

A spectral triple that gives rise to Poincar\'e duality in KK-theory first appeared in Kasparov \cite{Kasparov}. Kasparov's fundamental class, namely the Dirac element in Definition-Lemma 4.2 in \cite{Kasparov}, is essentially a spectral triple, although there was no such terminology at that time. The algebra underlying Kasparov's spectral triple is noncommutative and $\Z_2$-graded, but in many situations things would become much less complicated if the algebra were ungraded, especially when considering its Dixmier-Douady class or passing it to cyclic cohomology class via Connes-Chern character.
The projective spectral triple constructed in this paper (Corollary \ref{projectivesptr})
has a noncommutative but ungraded algebra, and we show in Section \ref{Kasparov} that it is in fact Morita equivalent to that of Kasparov's. To construct such a spectral triple, we first review in Section \ref{spectralanalysis} the definition of spectral triples with certain smoothness property, then introduce in Section \ref{morita} the notion of Morita equivalence between them, and then find in Section \ref{glue} that local spin spectral triples on small open subsets of a manifold can be glued together, via Morita equivalence, to form a globally defined spectral triple.

The noncommutative algebras underlying projective spectral triples are examples of Azumaya algebras. In Section \ref{Azumaya} we review some basic theory on  Azumaya algebras, such as the fact that Morita equivalent classes of Azumaya algebras are classified by their Dixmier-Douady classes, and that the K-theory of an Azumaya algebra $\mathbf A$ coincides with the twisted K-theory of the underlying manifold with the Dixmier-Douady class of $\mathbf A$.

Mathai-Stevenson \cite{Mathai-Stevenson} showed that the K-theory (tensoring with $\C$) of an Azumaya algebra $\mathbf A$ is isomorphic to the periodic cyclic homology group of $\mathbf A$ via Connes-Chern character, and that the latter is isomorphic to the twisted de-Rham cohomology of the underlying manifold with the Dixmier-Douady class of $\mathbf A$ via a generalized Connes-Hochschild-Kostant-Rosenberg (CHKR) map.
$$
\xymatrix{
K^0(M,\delta(\mathbf A))\otimes \C \ar[rr]^{\ch}_{\cong} \ar[rd]_{\ch_{\delta(\mathbf A)}}^{\cong} &
 &
HP_0(\mathbf A) \ar[ld]^{\Chkr}_{\cong}\\
&H^{\ev}_{\dR}(M,\delta(\mathbf A))}
$$
In Section \ref{CHKR}, we find an alternative CHKR map (Theorem \ref{alternative}) for the special case that the Dixmier-Douady class of $\mathbf A$ is torsion.

For each algebra $\A$, a finite projective $\A$-module $\E$ as a K-cocycle in the K-theory of $\A$ has a Connes-Chern character $\ch([\E])$ as a cyclic homology class, whereas a spectral triple on $\A$ as a K-cycle in the K-homology group of $\A$ also has a Connes-Chern character as a cyclic cohomology class, and the index pairing of a K-cocycle and a K-cycle is identical to the index pairing of their Connes-Chern characters \cite{NDG,Connes}. The main purpose of this paper is to compute the Connes-Chern character of the projective spectral triple and identify it with the Poincar\'e dual of the A-hat genus of the underlying manifold.

In Section \ref{localformula}, with the help of the alternative CHKR map $\rho$ and by applying Poincar\'e duality, we obtain our main result, a local formula for the Connes-Chern character of the projective spectral triple for every even dimensional oriented closed Riemannian manifold.

\noindent\textbf{Acknowledgements.}  I am very grateful to
Bai-Ling Wang. He foresaw the possibility that the projective spin
Dirac operator defined by \cite{Mathai} in formal sense can be
realized by a certain spectral triple, and introduced his
interesting research project to me in 2008. The spectral triple in
his mind turned out to be the projective spectral triple constructed
in this paper. Without his insight, I wouldn't have been writing
this thesis.
I would like to thank Adam Rennie for his very helpful remarks about Morita equivalence between various presentations of Kasparov's fundamental class on reading a draft of this paper.
I also wish to thank my advisor, Matilde
Marcolli, for her many years of encouragement, support, and many
helpful suggestions on both this research and other aspects.
\begin{section}{Azumaya bundles, twisted K-theory, and twisted cohomology}\label{Azumaya}

Suppose $X$ is a closed oriented manifold. Let
$$\M_n=\left\{
\begin{array}{ll}
\M_n(\C), & n=1,2,\dots  \\
\K(H), & n=\infty
\end{array}\right. ,
\quad
\U_n=\left\{
\begin{array}{ll}
\U(n), & n=1,2,\dots  \\
\U(H), & n=\infty
\end{array}\right.,$$
where $n$ could be either a positive integer or infinity, $H$ is an infinite dimensional separable Hilbert space, $\K(H)$ is the $C^*$-algebra of compact operators on $H$, and $\U(H)$ is the topological group of unitary operators with the operator norm topology. Kuiper's theorem states that $\U(H)$ is contractible. Let $\PU_n=\U_n/\U(1)$ be the projective unitary groups. In particular $\PU(H)=\PU_{\infty}$ is endowed with the topology induced from the norm topology of $\U(H)$.

Let $\Aut(\M_n)$ be the group of automorphisms of the $C^*$-algebra $\M_n$.

\begin{Fact}For every element $g\in \Aut(\M_n)$, there exists $\tilde g\in \U_n$, such that $g =\Ad \tilde g$. For every $u\in \U_n$, $\Ad\ u=1$ if and only if $u$ is scalar. In other words, as groups
$\PU_n\cong \Aut(\M_n).$
\end{Fact}
\begin{Fact}\label{torsion}If $n$ is finite, $\U_n/\U(1)\cong \mathrm{SU}(n)/\{z\in \C \mid z^n=1\}$.
\end{Fact}

\begin{defn}An \definition{Azumaya bundle} over $X$ of rank $n$ (possibly $n=\infty$) is a vector bundle over $X$ with fibre $\M_n$ and structure group $\PU_n$.
\end{defn}

Every Azumaya bundle of rank $n$ is associated with a principal $\PU_n$-bundle and vice versa.

\begin{defn}
The space $\mathbf{A}=\Gamma^0(A)$ of continuous sections of an Azumaya bundle $A$ over $X$ forms a $C^*$-algebra called an \definition{Azumaya algebra} over $X$.
\end{defn}

The following are examples of Azumaya algebras over $X$:
\begin{enumerate}
\item the algebra of complex valued continuous functions $C^0(X)$;
\item $C^0(X)\otimes \M_n$;
\item if $E$ is a finite rank vector bundle over $X$, the algebra of continuous sections of $\End(E)$, $\Gamma^0(\End(E))$;
\item if $X$ is an even dimensional Riemannian manifold, the algebra of continuous sections of the Clifford bundle $\Cl(T^*X)$, $\Gamma^0(\Cl(T^*X))$;
\item if $E$ is a real vector bundle over $X$ of even rank with a fiberwise inner product, the algebra of continuous sections of the Clifford bundle $\Cl(E)$, $\Gamma^0(\Cl(E))$.
\end{enumerate}
Note that examples i), ii) and iii) are Morita equivalent (in the category of C$^*$-algebras, i.e., strongly Morita equivalent \cite{Rieffel, ContinuousTrace}) to $C^0(X)$, while example iv) or v) is Morita equivalent to $C^0(X)$ if and only if $X$ or $E$ is spin$^c$ respectively.
\begin{Fact}The center of a finite Azumaya algebra over $X$ is $C^0(X)$.
\end{Fact}
\begin{Fact}An Azumaya algebra $\mathbf A$ over $X$ is locally Morita equivalent to $C^0(X)$.
\end{Fact}

The obstruction to an Azumaya algebra being (globally) Morita equivalent to its ``center'' is characterized by its Dixmier-Douady class:

\begin{defn}For every Azumaya bundle $\pi:A \to X$ of rank $n$, there is a cohomology class $\delta(A)$ in $H^3(X,\Z)$, called the \definition{Dixmier-Douady class} of $A$, constructed as follows:

Let $\{U_i\}_{i\in I}$ be a good covering of $X$, and write $U_{i_1\cdots i_n}$ for the intersection of $U_{i_1},U_{i_2},\cdots,U_{i_n}$. Suppose
$$\psi_i: U_i\times \M_n \to \pi^{-1}(U_i), \quad \forall i\in I ,$$
provide a local trivialization of $A$.
Then $\psi_i^{-1}\psi_j: U_{ij}\times \M_n \to U_{ij}\times \M_n$ give rise to the transition functions $g_{ij}\in C^0(U_{ij}, \Aut(\M_n))$.
Pick $\tilde g_{ij}\in C^0(U_{ij},\U_n)$ such that $\Ad\tilde g_{ij}=g_{ij}$ and $\tilde g_{ij}=\tilde g_{ji}^{-1}$.
Thus $\Ad (\tilde g_{ij}\tilde g_{jk}\tilde g_{ki}) = g_{ij}g_{jk}g_{ki} =1$, which implies
$$\mu_{ijk}:=\tilde g_{ij}\tilde g_{jk}\tilde g_{ki} \in C^0(U_{ijk}, \U(1)).$$
Therefore $\mu$ is a \v Cech $2$-cocycle with coefficient sheaf $\mathscr U(1):U\mapsto C^0(U,\U(1))$, since
$(\partial \mu)_{ijkl}=\mu_{jkl}\mu_{ikl}^{-1}\mu_{ijl}\mu_{ijk}^{-1}=1$. The \v Cech $2$-cocycle $\mu$ is also called the {\em bundle gerbe} structure of $A$. The short exact sequence of sheaves
$$0 \longrightarrow \Z \longrightarrow \mathscr R \xrightarrow{\exp 2\pi i\cdot} \mathscr U(1) \longrightarrow 0,$$
where $\mathscr R$ is the sheaf $U\mapsto C^0(U,\R)$, induces an isomorphism of \v Cech cohomology groups
$$\partial: \check H^2(X,\mathscr U(1))\xrightarrow{\cong} \check H^3(X,\Z).$$
Define the Dixmier-Douady class by $\delta(A):=\partial [\mu]$. More explicitly, pick $\nu_{ijk}\in C^0(U_{ijk},\R)$ such that $$\exp 2\pi i \nu_{ijk}=\mu_{ijk}.$$
Then $\exp 2\pi i (\partial \nu)_{ijkl} = (\partial \mu)_{ijkl}=1$, which implies $(\partial \nu)_{ijkl}=\nu_{jkl}-\nu_{ikl}+\nu_{ijl}-\nu_{ijk} \in \Z$ are locally constant integers on $U_{ijkl}$.
Thus $\delta(A)=[\partial \nu]\in \check H^3(X,\Z)$.
\end{defn}
\begin{defn}
Suppose that $A$ is an Azumaya bundle, that $\mathbf A$ is the Azumaya algebra corresponding to $A$, and that $P$ is the principal $\PU$-bundle associated to $A$. We say $\delta(\mathbf A)= \delta(P)=\delta(A)$ are the Dixmier-Douady class of $\mathbf A$ and $P$ respectively.
\end{defn}
As a consequence of Kuiper's theorem,
\begin{prop}For every cohomology class $\delta$ in $H^3(X,\Z)$, there is a unique $($up to isomorphism$)$ infinite rank Azumaya bundle $($or algebra$)$ with Dixmier-Douady class $\delta$.
\end{prop}
\begin{prop}
Let $A$ be an Azumaya bundle. If $\delta(A)=0$, then one can choose $\tilde g_{ij}$ so that $\tilde g_{ij}$ are the transition functions of a certain Hermitian bundle $E$ over $X$, and $A$ is isomorphic to $\K(E)$, the bundle over $X$ with fibres $\K(E_x)$ for all $x\in X$.
\end{prop}
\begin{cor}An Azumaya algebra $\mathbf A$ over $X$ is Morita equivalent to $C^0(X)$ if and only if $\delta(\mathbf A)=0$.
\end{cor}
\begin{cor}Two Azumaya algebras $\mathbf A_1$ and $\mathbf A_2$ over $X$ are Morita equivalent if and only if $\delta(\mathbf A_1)=\delta(\mathbf A_2)$. Namely, Morita equivalence classes of Azumaya algebras are parameterized by $H^3(X,\Z)$.
\end{cor}
As a consequence of \textbf{Fact \ref{torsion}},
\begin{prop}If $A$ is an Azumaya bundle of finite rank $n$, then $n\delta(A)=0$.
\end{prop}

For example, suppose $X$ is a $2m$-dimensional smooth manifold. The Clifford bundle $\Cl(T^*X)$ is an Azumaya bundle of rank $2^m$. Its Dixmier-Douady class $\delta(\Cl(T^*X))=W_3(X)$ is the third integral Stiefel-Whitney class of $X$, and $2W_3(X)=0$.
\begin{prop}If $A_1$ and $A_2$ are two Azumaya bundles over $X$, then $$\delta(A_1\otimes A_2)=\delta(A_1)+\delta(A_2).$$
\end{prop}
\begin{prop}If $\mathbf A$ is an Azumaya algebra, then its opposite algebra $\mathbf A^{\mathrm{op}}$ is also an Azumaya algebra and $$\delta(\mathbf A^{\mathrm{op}})=-\delta(\mathbf A).$$
\end{prop}

Let $\delta$ be a cohomology class in $H^3(X,\Z)$. Recall that (Rosenberg \cite{Rosenberg}, Atiyah-Segal \cite{AtiyahSegal}) the \definition{twisted K-theory} $K^0(X,\delta)$ can be defined by
$$K^0(X,\delta) =[P\to \mathrm{Fred}(H)]_{\PU(H)},$$
the abelian group of homotopy classes of maps $P\to \mathrm{Fred}(H)$ that are equivariant under the natural action of $\PU(H)$, where
$P$ is a principal $\PU(H)$-bundle over $X$ with Dixmier-Douady class $\delta(P)=\delta$; and where $\mathrm{Fred}(H)$ is the space of
Fredholm operators on $H$. Twisted K-theory can also be defined with K-theory of a C$^*$-algebra:
$$K^0(X,\delta)=K_0(\mathbf A),$$
where $\mathbf A$ is an (infinite rank) Azumaya algebra over $X$ with Dixmier-Douady class $\delta(\mathbf A)=\delta$.
One can also define the twisted $K^1$-group by $K^1(X,\delta)=K_1(\mathbf A)$.
The above two definitions of twisted K-theory are equivalent (Rosenberg
\cite{Rosenberg}). We will always use the second definition in this paper.

\begin{prop}The direct sum of twisted K-groups of $X$ $$\bigoplus_{\delta\in H^3(X,\Z)} K^\bullet(X,\delta)$$ forms a $\Z_2\times H^3(X,\Z)$-bigraded ring. The product $K^i(X,\delta_1)\times K^j(X,\delta_2)\to K^{i+j}(X,\delta_1+\delta_2)$ is naturally defined.
\end{prop}

\begin{defn} Let $c\in \Omega^3(X)$ be a closed $3$-form, the \definition{twisted de Rham complex} is the following periodic sequence
$$\xrightarrow{d_c} \Omega^{\ev}(X) \xrightarrow{d_c} \Omega^{\odd}(X) \xrightarrow{d_c},$$
where $d_c \omega=d\omega+c\wedge \omega$.
The \definition{twisted de Rham cohomology} is $H_{\dR}^*(X,c)=H^*(\Omega^*(X), d_c)$.
\end{defn}
\begin{prop}
If $c$ is a closed $3$-form, then $H_{\dR}^*(X,c)\cong H_{\dR}^*(X,zc)$ as isomorphic vector spaces for all nonzero $z\in \C$.
\end{prop}
In particular, $H_{\dR}^*(X,c)\cong H_{\dR}^*(X,-c)$ as vector spaces. In fact, in some literatures such as \cite{Mathai-Stevenson}, the twisted coboundary $d_c\omega$ of $\omega$ is defined by $d\omega- c \wedge \omega$.
\begin{prop}
If a closed $3$-form $c_1=c_2+d\beta$ for some $\beta\in \Omega^2(X)$, then
\begin{eqnarray*}
\CD
 @>>> \Omega^{\ev}(X)       @>d_{c_1}>>  \Omega^{\odd}(X) @>>> \\
  & & @VV \wedge \exp \beta V                  @VV \wedge \exp \beta  V    \\
 @>>> \Omega^{\ev}(X)       @>d_{c_2}>>  \Omega^{\odd}(X) @>>>
\endCD
\end{eqnarray*}
is a chain isomorphism. Therefore $H_{\dR}^*(X,c_1)\cong H_{\dR}^*(X,c_2)$ as vector spaces.
\end{prop}

\end{section} 
\begin{section}{Generalized Connes-Hochschild-Kostant-Rosenberg theorem}\label{CHKR}
In this section, we assume that $M$ is a smooth oriented closed manifold, and that $A$ is an Azumaya bundle over $M$ with a smooth structure in the sense that all the transition functions for the vector bundle $A$ are smooth functions valued in the (Banach) Lie group $\PU_n$. Let $\A$ be the space of smooth trace class sections of $A$, then $\A$ is a Fr\'echet pre-C$^*$-algebra densely embedded in $\mathbf A=\Gamma^0(A)$. In particular, if the rank $n$ of $A$ is finite, then $\A =\Gamma(A)$.

Given a $\PU_n$-connection $\nabla:\Omega^k(M,A)\to \Omega^{k+1}(M,A)$ on $A$, the image of the Dixmier-Douady class $\delta(A)$ in $H_{\mathrm{dR}}^3(M,\R)$ can be represented by a differential $3$-form in terms of  the connection and curvature (e.g., Freed \cite{Freed}) as follows:

Let $\{U_i\}$ be a good open covering of $M$, and $\psi_i:U_i\times \M_n \to A|_{U_i}$ be a local trivialization compatible with the smooth structure on $A$. Denote by $g_{ji}\in C^{\infty}(U_{ij}, \PU_n)$ the transition function corresponding to $\psi_j^{-1}\psi_i$. Pick $\tilde{g}_{ji}\in C^{\infty}(U_{ij}, \U_n)$ so that $\Ad\tilde{g}_{ji}=g_{ji}$ and $\tilde{g}_{ji}=\tilde{g}_{ij}^{-1}$. Let $\theta_i$ be the local connection forms of $\nabla$ on $U_i$,
$$\nabla(\psi_i(O))=\psi_i(dO+\theta_i(O)), \forall O\in C^{\infty}(U_i, \M_n).$$
Then $\theta_i=g^{-1}_{ji}\theta_jg_{ji}+g^{-1}_{ji}dg_{ji}$. Pick $\tilde{\theta}_i\in \Omega^1(U_i,\M_n)$ if $n\ne\infty$, or pick
$\tilde{\theta}_i\in \Omega^1(U_i,\mathrm B(H))$ if $n=\infty$, so that $\tilde{\theta}_i$ is skew-Hermitian and $\theta_i=\ad \tilde{\theta}_i$. Thus
$$\tilde{\theta}_i=\tilde{g}^{-1}_{ji}\tilde{\theta}_j\tilde{g}_{ji}+\tilde{g}^{-1}_{ji}d\tilde{g}_{ji}+\alpha_{ij},$$
for some scalar valued $1$-form $\alpha_{ij}\in \Omega^1(U_{ij})$.  Let $\omega_i$ be the local curvature forms of $\Omega=\nabla^2:\Gamma(A)\to\Omega^2(X,A)$ on $U_i$,
$$\Omega(\psi_i(O))=\psi_i(\omega_i(O)), \forall O\in  C^{\infty}(U_i, \M_n).$$
So $\omega_i=d\theta_i+\theta_i\wedge \theta_i$, and $\omega_i=g^{-1}_{ji}\omega_jg_{ji}$. Let $\tilde{\omega}_i=d\tilde{\theta_i}+\tilde{\theta}_i\wedge \tilde{\theta}_i$, then $\ad\tilde{\omega}_i=\omega_i$, and $\tilde{\omega}_i=\tilde g^{-1}_{ji}\tilde{\omega}_j\tilde g_{ji}+d\alpha_{ij}$.
Let $\tilde{\Omega_i}=\psi_i \tilde{\omega}_i \psi_i^{-1}$, then $\tilde{\Omega_i}=\tilde{\Omega_j}+d\alpha_{ij}$. Since $d\alpha_{ij}+d\alpha_{jk}+d\alpha_{ki}=0$, $d\alpha$ forms a $2$-form valued cocycle, and since the sheaf of $2$-forms is fine (or because of the existence of partition of unity on $M$), there exist $\beta_i\in \Omega^2(U_i)$ so that $2\pi i (\beta_i-\beta_j) = d\alpha_{ij}$. We can define a generalized $2$-form by $\tilde{\Omega_i}-2\pi i \beta_i$ on $U_i$, and it is globally well-defined.

\begin{lem}If $A$ is an Azumaya bundle over $M$ with connection $\nabla$ and curvature $\Omega$, then
$-{1 \over 2\pi i}\nabla(\tilde{\Omega_i}-2\pi i \beta_i)$ represents the image of the Dixmier-Douady class $\delta(A)$ in $H^3_{\dR}(M)$.
\end{lem}
\begin{proof}First recall that the \v Cech-de Rham isomorphism between the third de Rham cohomology $H^3_{\dR}(M)$ and  \v Cech cohomology $\check H^3(M,\C)$ with constant coefficient sheaf $\C$ can be constructed as follows.
For any closed $3$-form $\tau\in \Omega^3(M)$, one can find $\beta(\tau)_i\in \Omega^2(U_i)$ so that $d\beta(\tau)_i=\tau|_{U_i}$. Since $d\beta(\tau)_i-d\beta(\tau)_j=0$ one can find $\alpha(\tau)_{ij}\in \Omega^1(U_{ij})$ so that $\beta(\tau)_i-\beta(\tau)_j=d\alpha(\tau)_{ij}$. Since $d\alpha(\tau)_{ij}+d\alpha(\tau)_{jk}+d\alpha(\tau)_{ki}=0$ one can find $\nu(\tau)_{ijk}\in C^{\infty}(U_{ijk})$ so that $(\partial\alpha(\tau))_{ijk}=d\nu(\tau)_{ijk}$ on $U_{ijk}$. Here $\partial$ denotes the coboundary operator on \v{C}ech cocycles. Likewise, since $d(\partial \nu(\tau))_{ijkl}=(\partial d\nu(\tau))_{ijkl}=0$, one can find $\delta(\tau)_{ijkl}\in \C$ so that $(\partial \nu(\tau))_{ijkl}=\delta(\tau)_{ijkl}$. The \v Cech-de Rham isomorphism $H^3_{\dR}(M)\to \check H^3(M,\C)$ is given by the correspondence $\tau\mapsto \delta(\tau)$.

Now let $\tau=-{1\over 2\pi i}\nabla(\tilde{\Omega_i}-2\pi i \beta_i)$. By the Bianchi identity $-{1\over 2\pi i}\nabla(\tilde{\Omega_i}-2\pi i \beta_i)= d\beta_i$. Thus we can choose $\beta(\tau)_i=\beta_i$, $\alpha(\tau)_{ij}=\alpha_{ij}, \nu(\tau)_{ijk}=\nu_{ijk}$, and $\delta(\tau)_{ijkl}=\delta_{ijkl}$. Therefore it follows that $\tau$ represents the image of $\delta(A)$ in $H^3_{\dR}(M)$.
\end{proof}

\begin{thm}\label{sigma}$1)$. If $A$ is a finite rank Azumaya bundle with connection $\nabla$ and curvature $\Omega$, then there is a unique traceless $\sigma(\Omega)\in \Omega^2(M, A)$ such that $\ad\, \sigma(\Omega)=\Omega$ and $\nabla(\sigma(\Omega))=0$.

\noindent $2)$. Suppose that $A$ is an infinite rank Azumaya bundle associated to a principal $\PU(H)$-bundle $P$ with connection $\nabla$ and curvature $\Omega$, and that $c\in \Omega^3(M)$ is a differential form representing the image of the Dixmier-Douady class $\delta(A)$ in $H^3_{\dR}(M)$. Then up to a closed scalar $2$-form, there is a unique $\Gamma(P\times_{\PU(H)}\mathrm B(H))$-valued $2$-form $\sigma(\Omega)$ such that $\ad\, \sigma(\Omega)=\Omega$ and $-{\nabla(\sigma(\Omega))\over 2\pi i}=c$. Here $\PU(H)$ acts on $\mathrm B(H)$ the same way as on $\K(H)$.
\end{thm}
\begin{proof}Up to a scalar valued $2$-form, $\sigma(\Omega)$ can be defined by $\tilde{\Omega_i}-2\pi i \beta_i$ as in the above lemma.
In fact, $\ad (\tilde{\Omega_i}-2\pi i \beta_i)=\ad \tilde{\Omega_i}=\Omega$, and  $-{1 \over 2\pi i}\nabla(\tilde{\Omega_i}-2\pi i \beta_i)$ represents the image of the Dixmier-Douady class $\delta(A)$ in $H^3_{\dR}(M)$.

\noindent 1). If $A$ is of rank $m<\infty$, we just take $\sigma(\Omega)=\tilde{\Omega_i}-2\pi i \beta_i- \tr (\tilde{\Omega_i}-2\pi i \beta_i)/m$. The scalar $3$-form $\nabla(\sigma(\Omega))$ must be $0$ since it is traceless. The uniqueness of $\sigma(\Omega)$ is obvious.

\noindent 2). Suppose $A$ is of infinite rank. There exists a scalar $2$-form $\eta$ such that $$-{1 \over 2\pi i}\nabla(\tilde{\Omega_i}-2\pi i \beta_i)-c =d \eta.$$
We can take $\sigma(\Omega)=\tilde{\Omega_i}-2\pi i \beta_i-2\pi i \eta$.
\end{proof}

Recall that if $\B$ is a pre-C$^*$-algebra densely embedded in a C$^*$-algebra $\mathbf B$, then $K_0(\B)=K^{\text{alg}}_0(\B)$ is naturally isomorphic to $K_0(\mathbf B)$. If $\B$ is a unital Fr\'echet algebra, we define $K_1(\B)$ to be the abelian group of the equivalence classes of $\mathrm{GL}_{\infty}(\B)$. We say that $u,v\in \mathrm{GL}_{\infty}(\B)$ are equivalent if there is a piecewise $C^1$-path in $\mathrm{GL}_{\infty}(\B)$ joining $u$ and $v$. The definition of $K_1(\B)$ can be extended to the case of non-unital algebras so that the six-term exact sequence property holds. For Azumaya algebras, $K_*(\A)$ is naturally isomorphic to $K_*(\mathbf A)=K^*(M, \delta(\mathbf A))$. We refer to \cite{NDG,Connes,Loday} for the definitions of Hochschild, cyclic and periodic cyclic homologies and cohomologies. We denote the Hochschild boundary map by $\bb$, and denote the Connes boundary map by $\BB$.

\begin{defn}Following Gorokhovsky \cite{Gorokhovsky}, define two maps
$$\Chkr: \bigoplus_{k\, \mathrm{even}}\overline{C}^{\mathrm{red}}_k(\A) \to \Omega^{\ev}(M)\quad\text{and}\quad
\Chkr: \bigoplus_{k\, \mathrm{odd}}\overline{C}^{\mathrm{red}}_k(\A) \to \Omega^{\odd}(M)
$$
by the JLO-type (\cite{JLO}) formula
\begin{equation}\label{Gorokhovsky}\Chkr(a_0,a_1,...,a_k)=\int_{\boldsymbol s\in\Delta^k} \tr(a_0 e^{-s_0\sigma(\Omega)}(\nabla a_1)e^{-s_1\sigma(\Omega)}\cdots(\nabla a_k)e^{-s_k\sigma(\Omega)}) d\boldsymbol s.
\end{equation}
Here  $\overline{C}^{\mathrm{red}}_0(\A)=\A$ and $\overline{C}^{\mathrm{red}}_j(\A)=\A^+\hat{\otimes}\A^{\hat{\otimes} j}$, for all $j\ne 0$, with $\A^+$ being the unitalization of $\A$ and $\hat{\otimes}$ the projective tensor product of locally convex topological algebras.
Here and subsequently, if $V=\bigoplus_i V^i$ is a $\Z$-graded vector space, we write $V^{\ev}=\bigoplus_k V^{2k}$ and $V^{\odd}=\bigoplus_k V^{2k+1}$.
\end{defn}

The generalized CHKR theorem of Mathai-Stevenson's is as follows.
\begin{prop}[Mathai-Stevenson \cite{Mathai-Stevenson}]$1)$. The map $\Chkr$ in {\rm(\ref{Gorokhovsky})} induces a quasi-isomorphism between the two complexes $$\xrightarrow{\bb}\overline{C}^{\mathrm{red}}_{\ev}(\A)\xrightarrow{\bb}\overline{C}^{\mathrm{red}}_{\odd}(\A)\xrightarrow{\bb}
\text{ and }
\xrightarrow{0}\Omega^{\ev}(M)\xrightarrow{0}\Omega^{\odd}(M)\xrightarrow{0},
$$
and hence isomorphisms $HH_{\ev}(\A)\cong \Omega^{\ev}(M)$ and $HH_{\odd}(\A)\cong \Omega^{\odd}(M)$.

$2)$. The map $\Chkr$ induces a quasi-isomorphism between the complex
$$\xrightarrow{\bb+\BB}\overline{C}^{\mathrm{red}}_{\ev}(\A)\xrightarrow{\bb+\BB}\overline{C}^{\mathrm{red}}_{\odd}(\A)\xrightarrow{\bb+\BB},
 \text{ or equivalently }
\xrightarrow{\BB}HH_{\ev}(\A)\xrightarrow{\BB}HH_{\odd}(\A)\xrightarrow{\BB}
$$
and the twisted de Rham complex
$$ \xrightarrow{d_c}\Omega^{\ev}(M)\xrightarrow{d_c}\Omega^{\odd}(M)\xrightarrow{d_c};$$
and hence an isomorphism
\begin{equation}\label{Ch}\Chkr: HP_{\bullet}(\A)\xrightarrow{\cong} H_{\dR}^{\bullet}(M,c), \quad ({\bullet}=\ev, \odd),
\end{equation}
 where $c=-{\nabla(\sigma(\Omega))\over 2\pi i}$ is a representative of the image of $\delta(A)$ in $H_{dR}^3(M)$.

$3)$. The Connes-Chern character $\ch:K_{\bullet}(\A)\otimes \C \to HP_{\bullet}(\A)$ and the {\em \bf twisted Chern character}
$$\ch_{\delta(A)}=\Chkr \circ \ch:\, K^{\bullet}(M,\delta(A))\otimes \C \to H_{\dR}^{\bullet}(M,c)$$
are isomorphisms.
\end{prop}

If $\delta(\A)$ is a torsion class, then there is an alternative formula for the map $\Chkr$, which we will see, is closely related to the relative Chern character (\cite{GetzlerBook}) of Clifford modules. From now on we take $c=0$, then by Theorem \ref{sigma}, up to a closed $2$-form, there is a unique $\sigma(\Omega)$ such that $\ad \sigma(\Omega)=\Omega$ and $\nabla(\sigma(\Omega))=0$.

Define $\psi_k:\A ^{\hat{\otimes} k}\to \A\hat{\otimes}_{\smooth(M)}\Omega^{k}(M)$ by letting
$$\psi_{-1}=0, \ \psi_0=1,\ \psi_1(a_1)=\nabla a_1, \ \psi_2(a_1,a_2)=(\nabla a_1)(\nabla a_2)+a_1 \sigma(\Omega)a_2,$$
\begin{equation}\psi_k(a_1,...,a_k)=(\nabla a_1) \psi_{k-1}(a_2,...,a_k) + a_1 \sigma(\Omega)a_2 \psi_{k-2}(a_3,...,a_k),\ \forall k\ge 2.\end{equation}
In other words, $\psi_k(a_1,...,a_k)$ is obtained as follows: Consider all partitions $\pi$ of the ordered set $\{a_1,...,a_k\}$ into blocks, where each block contains either one or two elements. Assign to each block $\{a_i\}$ of $\pi$ a term of the form $\nabla a_i$, and to each block $\{a_j,a_{j+1}\}$ of $\pi$ a term of the form $a_j\sigma(\Omega)a_{j+1}$. Then let $\psi_{k,\pi}$ be the product of these terms, and $\psi_k(a_1,...,a_k)$ be the sum of $\psi_{k,\pi}$ over all such partitions.
So in its expansion, $\psi_k(a_1,...,a_k)$ consists of a Fibonacci number of summands. Then define $\rho_k: \A ^{\hat{\otimes} k+1}\to \Omega^k(M)$ by
\begin{equation}\label{rho}\rho_k(a_0,...,a_k)={1\over k!}\tr(a_0 \psi_k(a_1,...,a_k))
\end{equation}
for all $k\ge 0$. Note that $\rho$ depends on the choice of connection $\nabla$ and, for infinite rank Azumaya bundles, the choice of $\sigma(\Omega)$, so we will write the complete form $\rho^{\nabla}_{\sigma}$ for $\rho$ when we need to specify $\nabla$ and $\sigma$.
We have the following results about $\rho$.
\begin{lem}$\label{LEMMAA}\rho \circ \bb = 0$.
\end{lem}

\begin{lem}\label{LEMMAX}If $\delta(\A)$ is a torsion class, then for all $k\ge 0$ and $a_i\in \A$,
$$(-1)^{k-1}\rho_k(a_0,\dots, a_{k})  + \rho_k(a_{k},a_0,\dots,a_{k-1})  = d\, \tr \big( a_0\, \psi_{k-1}(a_1,...a_{k-1})\, a_{k} \big).$$
\end{lem}
\begin{proof}Noticing that $d\circ \tr= \tr \circ \nabla$, it suffices to show
\begin{equation}\label{induction} (-1)^{k-1} a_0\psi_k(a_1,..,a_k) +\psi_k(a_0,...,a_{k-1})a_k = \nabla\big(a_0 \psi_{k-1}(a_1,...,a_{k-1}) a_k\big),
\end{equation}
for all $a_i\in \A^+$.
In fact, it is easy to see (\ref{induction}) is true for $k=0,1,2$. Suppose (\ref{induction}) holds for all $k\le m$ for some $m$. Then using  $\nabla^2= \ad \sigma(\Omega)$ and the Bianchi identity $\nabla (\sigma(\Omega)) =0$, we have
\begin{eqnarray*}
\lefteqn{\nabla(a_0\psi_m(a_1,...,a_m) a_{m+1})}\\
 & =& \nabla a_0 \psi_m(a_1,...,a_m) a_{m+1} + a_0 \nabla \psi_m(a_1,..., a_m) a_{m+1} \\
& & + (-1)^m a_0 \psi_m(a_1,...,a_m) \nabla a_{m+1}\\
& =&  \nabla a_0 \psi_m(a_1,...,a_m) a_{m+1} + a_0 \nabla \big( \nabla a_1 \psi_{m-1}(a_2,...,a_m)\big) a_{m+1} \\
& & + a_0 \nabla \big(a_1 \sigma(\Omega) a_2 \psi_{m-2}(a_3,...,a_m) \big) a_{m+1}  + (-1)^m a_0 \psi_m(a_1,...,a_m)\nabla a_{m+1} \\
& =& \psi_{m+1}(a_0,..., a_m) a_{m+1} -a_0 a_1 \sigma(\Omega) \psi_{m-1}(a_2,...,a_m) a_{m+1} \\
& & -a_0 \nabla a_1 \nabla \psi_{m-1}(a_2,...,a_m) a_{m+1}  + a_0 \nabla a_1 \sigma(\Omega) a_2 \psi_{m-2}(a_3,...,a_m) a_{m+1}\\
& & + a_0 a_1 \sigma(\Omega) \nabla a_2 \psi_{m-2}(a_3,...,a_m) a_{m+1} + a_0 a_1 \sigma(\Omega) a_2 \nabla \psi_{m-2}(a_3,...,a_m) a_{m+1}\\
& & + (-1)^m a_0 \psi_m(a_1,...,a_m)\nabla a_{m+1} \\
& =& \psi_{m+1}(a_0,..., a_m) a_{m+1} - a_0 a_1 \sigma(\Omega) \psi_{m-1}(a_2,...,a_m) a_{m+1} \\
& & +(-1)^m a_0 \nabla a_1 \psi_{m}(a_2,...,a_m,1) a_{m+1} - a_0 \nabla a_1 \psi_{m}(1, a_2,...,a_m) a_{m+1}\\
& & +a_0 \nabla a_1 \sigma(\Omega) a_2 \psi_{m-2}(a_3,...,a_m) a_{m+1} + a_0 a_1 \sigma(\Omega) \nabla a_2 \psi_{m-2}(a_3,...,a_m) a_{m+1}\\
& &+ (-1)^m a_0 a_1 \sigma(\Omega) a_2  \psi_{m-1}(a_3,...,a_m,1) a_{m+1} \\
& &+  a_0 a_1 \sigma(\Omega) a_2 \psi_{m-1}(1,a_3,...,a_m) a_{m+1}\\
& & + (-1)^m a_0 \psi_m(a_1,...,a_m)\nabla a_{m+1}\\
& =& \psi_{m+1}(a_0,..., a_m) a_{m+1} - a_0 a_1 \sigma(\Omega) \psi_{m-1}(a_2,...,a_m) a_{m+1} \\
& & +(-1)^m a_0 \nabla a_1 \psi_{m-2}(a_2,...,a_{m-1}) a_m \sigma(\Omega) a_{m+1}  \\
& & +a_0 a_1 \sigma(\Omega) \nabla a_2 \psi_{m-2}(a_3,...,a_m) a_{m+1}\\
& &+ (-1)^m a_0 a_1 \sigma(\Omega) a_2  \psi_{m-3}(a_3,...,a_{m-1})a_m \sigma(\Omega) a_{m+1}\\
& & +  a_0 a_1 \sigma(\Omega) a_2 \sigma(\Omega) a_3 \psi_{m-3}(a_4,...,a_m) a_{m+1}\\
& & + (-1)^m a_0 \psi_m(a_1,...,a_m)\nabla a_{m+1}\\
& =& \psi_{m+1}(a_0,..., a_m) a_{m+1} +(-1)^m a_0 \nabla a_1 \psi_{m-2}(a_2,...,a_{m-1}) a_m \sigma(\Omega) a_{m+1}  \\
& &+ (-1)^m a_0 a_1 \sigma(\Omega) a_2  \psi_{m-3}(a_3,...,a_{m-1})a_m \sigma(\Omega) a_{m+1}\\
& & + (-1)^m a_0 \psi_m(a_1,...,a_m)\nabla a_{m+1}\\
& =& \psi_{m+1}(a_0,..., a_m) a_{m+1}+ (-1)^m a_0 \psi_{m+1}(a_1,..., a_{m+1}).
\end{eqnarray*}
Thus identity (\ref{induction}) is proved by induction.
\end{proof}
\begin{lem}\label{LEMMAB} If $\A$ is unital, then $\rho_{k+1} \circ \BB_k=d \circ \rho_k$ on $\A ^{\hat{\otimes} k+1}$.
\end{lem}
\begin{proof} Recall that
$$\BB=(1-\lambda)\circ s \circ N,$$
where $\lambda$ is the cyclic permutation
$$\lambda(a_0,...,a_n)=(-1)^n (a_n,a_0,...,a_{n-1}),$$
$N$ is the sum of all cyclic permutations
$$N_k=1+\lambda+\cdots+\lambda^k,$$ and $s_k:C_k(\A)\to C_{k+1}(\A)$ is the extra degeneracy operator
$$s(a_0,...,a_k)= (1, a_0,...,a_k).$$
 Thus we have
\begin{eqnarray*}\rho \circ \BB (a_0,...,a_k) & = & \rho \circ (1-\lambda) \circ \sum_{i=0}^{k} (-1)^{ki}(1,a_i,...,a_{i-1})\\
(\text{By Lemma \ref{LEMMAX}.})&=&{1\over (k+1)!} \sum_{i=0}^k (-1)^{ki}(-1)^k d \tr(\psi(a_i,...,a_{i-2})a_{i-1})\\
&=& {1\over (k+1)}\sum_{i=0}^k (-1)^{k(i-1)} d \rho(a_{i-1},...,a_{i-2})\\
(\text{By Lemma \ref{LEMMAX} again.})&=& d \circ \rho (a_0,...,a_k).
\end{eqnarray*}
\end{proof}

\begin{cor}\label{CORB}$\rho_{k+1} \circ \BB_k=d \circ \rho_k$ on $\overline{C}^{\mathrm{red}}_k(\A)$.
\end{cor}
We now take a look at what happens for formula (\ref{rho}) when we choose different connections and different $\sigma$'s. Let $\nabla_t=\nabla_0+t \theta$ be a family of $\PU_n$-connections on the Azumaya bundle $A$, where $t\in I=[0,1]$, and $\theta$ is a generalized $1$-form. Denote by $\Omega_t$ the curvature of $\nabla_t$. If $A$ has finite rank, then the choice of $\sigma$ is unique, and $\sigma(\Omega_t)$ is a smooth family of traceless generalized $2$-forms. If $A$ has infinite rank, then given $\sigma_0(\Omega_0)$ and $\sigma_1(\Omega_1)$ with $\nabla_i(\sigma_i(\Omega_i))=0$ and $\ad \sigma_i(\Omega_i)=\Omega_i$, where $i=0,1$, we can always find a smooth family $\sigma_t(\Omega_t)$ with $\nabla_t(\sigma_t(\Omega_t))=0$ and $\ad \sigma_t(\Omega_t)=\Omega_t$. In fact,
 consider the projection $\widetilde{\pi}:M\times I \to I$ and the pull-back bundle $\widetilde{\pi}^*A=A\times I$ over $M\times I$ with connection $\widetilde{\nabla}=
\nabla_t + dt \partial_t$. The curvature on $\widetilde{\pi}^* A$ is $\widetilde{\Omega}=\Omega_t+\theta dt$.  Since $\widetilde{\pi}^*A$ is also an Azumaya bundle with torsion Dixmier-Douady class,
we can choose some $\sigma(\widetilde{\Omega})$ with $\widetilde{\nabla}(\sigma(\widetilde{\Omega}))=0$ and $\ad \sigma(\widetilde{\Omega})=\widetilde{\Omega}$. Suppose
 $\sigma(\widetilde{\Omega})=\Lambda_t+\eta_t dt,$
where $\Lambda_t$ and $\eta_t$ are a smooth family of generalized $2$-forms and a smooth family of generalized $1$-forms on $M$, then
$ \ad \Lambda_t=\Omega_t$, $\ad \eta_t=\theta$, and $\nabla_t \Lambda_t=0$. Therefore, by Theorem \ref{sigma}, $\sigma_0(\Omega_0)-\Lambda_0$ and $\sigma_1(\Omega_1)-\Lambda_1$ are closed scalar $2$-forms.
Then we can take a smooth family of generalized $2$-forms
$$\sigma_t(\Omega_t)=\Lambda_t+(1-t)(\sigma_0(\Omega_0)-\Lambda_0)+t(\sigma_1(\Omega_1)-\Lambda_1).$$
We have $\nabla_t(\sigma_t(\Omega_t))=0$ and $\ad \sigma_t(\Omega_t)=\Omega_t$.

Consider the good open covering $\{U_i\}$ of $M$ again. Suppose $$(\sigma_0(\Omega_0)-\Lambda_0)-(\sigma_1(\Omega_1)-\Lambda_1)=d \eta_i$$ on $U_i$ for some scalar $1$-form $\eta_i$, and let
$$\widetilde{\sigma}(\widetilde{\Omega})=\sigma_t(\Omega_t)+\eta_t dt+\eta_i dt.$$
Then we have $\widetilde{\nabla}(\widetilde{\sigma}(\widetilde{\Omega}))=0$ and $\ad \widetilde{\sigma}(\widetilde{\Omega})=\widetilde{\Omega}$.

Following Mathai-Stevenson \cite{Mathai-Stevenson}, let
$$K(a_0,...,a_k)=\int_I dt\wedge(\iota_{\partial_t}\rho_{\widetilde{\sigma}}^{\widetilde{\nabla}}(a_0,...,a_k))$$
on $U_i$.
Here the functions $a_0,...,a_k$ inside $\rho_{\widetilde{\sigma}}^{\widetilde{\nabla}}()$ are considered as functions on $U_i\times I$ which are constant in the $t$ direction.
 We have the following lemma.

\begin{lem}\label{LEMMAC} The map $K$ defined above is a chain homotopy -- we have the formula
$$\rho^{\nabla_1}_{\sigma_1}-\rho^{\nabla_0}_{\sigma_0}=K\circ (\bb+\BB) - d \circ K$$
on $\overline{C}^{\mathrm{red}}_*(\A|_{U_i})$.
\end{lem}
\begin{proof}By Lemma \ref{LEMMAA} and Corollary \ref{CORB},
\begin{eqnarray*}K\circ (\bb+\BB) - d \circ K &=& K\circ \BB- d\circ K \\
& = & \int_I dt \wedge \iota_{\partial_t} \circ d \circ \rho_{\widetilde{\sigma}}^{\widetilde{\nabla}} -d \int_I dt \wedge \iota_{\partial_t} \circ \rho_{\widetilde{\sigma}}^{\widetilde{\nabla}} \\
& = & \int_I dt \wedge (\iota_{\partial_t}\circ  d + d \circ \iota_{\partial_t}) \circ \rho_{\widetilde{\sigma}}^{\widetilde{\nabla}}\\
&=& \int_I dt \wedge {\partial \over \partial t} \circ \rho_{\widetilde{\sigma}}^{\widetilde{\nabla}} \\
&=& \rho^{\nabla_1}_{\sigma_1}-\rho^{\nabla_0}_{\sigma_0}.
\end{eqnarray*}
\end{proof}
\begin{thm}\label{THMA} If $\delta(\A)$ is a torsion class, then the map $\rho_k$ in {\rm (\ref{rho})} induces a quasi-isomorphism between the complexes
$$\xrightarrow{\bb+\BB}\overline{C}^{\mathrm{red}}_{\ev}(\A)\xrightarrow{\bb+\BB}\overline{C}^{\mathrm{red}}_{\odd}(\A)\xrightarrow{\bb+\BB}
\text{ and } \xrightarrow{d}\Omega^{\ev}(M)\xrightarrow{d}\Omega^{\odd}(M)\xrightarrow{d},$$
and its induced isomorphism
\begin{equation}\rho: HP_{\bullet}(\A)\xrightarrow{\cong} H_{\dR}^{\bullet}(M)
\end{equation}
coincides with the map $\Chkr$ in {\rm(\ref{Ch})}.
\end{thm}
\begin{proof} First, by Lemma \ref{LEMMAA} and Corollary \ref{CORB}, we see that $\rho$ is a chain map. Then the homotopy formula in Lemma \ref{LEMMAC} and Mathai-Stevenson \cite{Mathai-Stevenson}'s spectral sequence argument on the \v{C}ech-de Rham bicomplex prove the theorem.
\end{proof}
Because the map $\rho$ is degree-preserving, we have
\begin{cor}If $\delta(\A)$ is a torsion class, then the map $\rho$ in {\rm(\ref{rho})} induces a chain isomorphism
$$(HH_*(\A), \BB)\rightarrow (\Omega^*(M), d).$$
\end{cor}
\begin{thm}\label{alternative} If $\delta(\A)$ is a torsion class, then the map $\rho$ in {\rm (\ref{rho})} induces a homomorphism
$$\rho: C^{\lambda}_*(\A)\to \Omega^*(M)/d(\Omega^{*-1}(M)),$$
{\em where $C^{\lambda}_*(\A)$ is the Connes complex of $\A$ (cf. \cite{NDG}, \cite{Connes})},
and an isomorphism $$\rho: HP_{\bullet}(\A)\xrightarrow{\cong} H_{\dR}^{\bullet}(M),$$
which coincides with $\Chkr$ in {\rm(\ref{Ch})}.
\end{thm}
\begin{proof}First, by Lemma \ref{LEMMAX} we see that the induced homomorphism $${C}^{\lambda}_*(\A) \, \to \Omega^*(M)/d(\Omega^{*-1}(M))$$ is well-defined.

 Next, we show that the induced map $\rho: HP_{\bullet}(\A)\to H_{\dR}^{\bullet}(M)$ is well-defined. Note that $\rho_k\circ \bb=0$ on $C_{k+1}(\A)$. This means that the map ${C}^{\lambda}_*(\A)/ \bb ({C}^{\lambda}_{*+1}(\A))$ $\to$ $\Omega^*(M)/d(\Omega^{*-1}(M))$ is well-defined. Then it suffices to show that the images of $HP_*(\A)$ under the map $\rho$ are represented by closed forms. We prove this only for the case of even degree, and the odd degree case is similar. Since $\ch: K_0(\A)\to HP_0(\A)$ is an isomorphism, elements of $HP_0(\A)$ are generated by $\ch[p]$ for $[p]\in K_0(\A)$ (here we may suppose $\A$ is unital for simplicity).  The periodic Connes-Chern character $\ch[p]$ is represented by a sequence of cyclic cycles $\{\ch^{\lambda}_0(p),\ch^{\lambda}_2(p),...\}$, where
$$\ch^{\lambda}_{2m}(p)=(-1)^m{(2m)!\over m!}\tr(p^{\otimes 2m+1}) \in C^{\lambda}_{2m}(\A).$$
Observe that $p(\nabla p)^{2i+1}p=0$ for all idempotent $p$ and $i\ge 0$, then
$$\rho_{2k}(\ch_{2k}^{\lambda}(p))=(-1)^k{(2k)!\over k!}\tr\big(p\ \psi_{2k}(p,...,p)\big)
=(-1)^k{(2k)!\over k!}\tr\big(p \psi_2(p,p)^k \big),$$
because in the expansion of $p\ \psi_{2k}(p,...,p)p$, any term that has a factor $p(\nabla p)^{2i+1}p$ vanishes. Since $\nabla(p \psi_2(p,p))=0$, it follows that $\rho(\ch[p])$ is a closed form for all $[p]\in K_0(\A)$.

Finally, since the periodic cyclic homology of $\A$ defined from the $C^{\lambda}$-complex is the same as that defined from the $(\bb,\BB)$-bicomplex, by Theorem \ref{THMA}, the induced homomorphism $\rho: HP_{\bullet}(\A)\to H_{\dR}^{\bullet}(M)$ is an isomorphism identical to $\Chkr$ in {\rm(\ref{Ch})}.
\end{proof}
\end{section} 
\begin{section}{Spectral analysis of spectral triples}\label{spectralanalysis}
In this section we review the definition and some analytical properties of spectral triples. Note that a slight modification to the standard definition of spectral triple (cf. \cite{CM}) is made so that it will be more convenient to develop the theory in this paper. In fact, in definition \ref{unital_sptr} we require that the second entry $\mathcal H$ of a spectral triple $(\A,\mathcal H, D)$ to be an $\A$-module as well as the smooth Sobolev domain of $D$, instead of the Hilbert space $\bar{\mathcal H}$, the norm completion of $\mathcal H$. So in application in differential geometry, spectral triples defined this way operate directly with smooth sections of vector bundles.  For a spectral triple in the conventional sense, that would be a strong requirement, as strong as the smoothness condition in Appendix B in \cite{ConnesMoscovici}.

Suppose that $D$ is a densely defined self-adjoint operator on a
Hilbert space $H$, and that $D$ has compact resolvent. Let
$\mu_1>\mu_2>\cdots$ be the list of eigenvalues of $(D^2+1)^{-1}$ in
decreasing order, and $V_i\subset H$ be the eigenspace corresponding
to $\mu_i$ for each $i$. Then every vector $v \in H$ can be uniquely
represented as a sequence $(v_1,v_2,\dots)$ with $v_i\in V_i$ and
$\sum_i\|v_i\|^2<\infty$, and vice versa.

For every $s \ge 0$, consider the following subspaces of $H$,
$$W^s(D)=\{(v_1,v_2,\dots)\in H \mid \sum_i \mu_i^{-s}\|v_i\|^2<\infty\},$$
with the norm $\|(v_1,v_2,\dots)\|_s=\sqrt{\sum_i
\mu_i^{-s}\|v_i\|^2}$;
$$W^{s,p}(D)=\{(v_1,v_2,\dots)\in H \mid \sum_i \mu_i^{-sp/2}\|v_i\|^p<\infty\}, \quad \forall p>0,$$
with the norm $\|( v_1,v_2,\dots)\|_{s,p}=\left(\sum_i
\mu_i^{-sp/2}\|v_i\|^p \right)^{1/p}$; and
$$W^{s,\infty}(D)=\{(v_1,v_2,\dots)\in H \mid \sup_i \mu_i^{-s/2}\|v_i\|<\infty\},$$
with the norm $\|( v_1,v_2,\dots)\|_{s,\infty}=\sup_i
\mu_i^{-s/2}\|v_i\|$. $W^s=W^{s,2}$ has a natural Hilbert space
structure.
\begin{prop}[Rellich] For each $\epsilon>0$, the inclusion $W^{s+\epsilon}\to W^s$ is
compact.
\end{prop}
\begin{prop} $W^1\subset H$ is the domain of the self-adjoint operator $D$, and $D:W^1\to H$ is a Fredholm operator.
\end{prop}
Let $W^{\infty}=\bigcap_{s>0}W^s$, then $W^{\infty}$ is a Fr\'echet
space with a family of norms $\|\cdot\|_s$. It is easy to see that
restricted to $W^{\infty}$, the mapping $D: W^{\infty} \to
W^{\infty}$ is continuous with respect to the Fr\'echet space
topology.

We say the operator $D$ is \definition{finitely summable} or has
\definition{spectral dimension} less than $2d$ (for some real number $d>0$), if
$(D^2+1)^{-d}$ is a trace class operator.

\begin{thm}\label{continuity}Suppose $D$ has finite spectral dimension.
If $T \in B(H)$ is a bounded operator that maps $W^{\infty}$ into
$W^{\infty}$, then the restricted mapping $T:W^{\infty}\to
W^{\infty}$ is also continuous.\end{thm} This theorem can be proved by
the following lemmas.
\begin{lem}[Sobolev embeddings]\label{estimate}If $D$ has spectral dimension less than
$2d$, then we have the following obvious estimate:
$$\|v\|_{s,\infty}\le \|v\|_{s,p} \le \bigg(\sum_j \mu_j^d\bigg)^{1/p} \|v\|_{s+{2d\over p},\infty},
\quad \forall v\in  H, \forall s\ge 0, \forall p>0,$$ i.e., there
are bounded embeddings $W^{s+{2d\over p},\infty}\subset W^{s,p}
\subset W^{s,\infty}$.
\end{lem}
\begin{lem}Suppose $D$ has finite spectral dimension.
\begin{enumerate}\item Let $T: W^{\infty}\to W^{\infty}$ be a continuous
operator. Suppose for each $j$,
$$T(0,\dots,0,v_j,0,\dots)=(t_{1j}v_j,t_{2j}v_j,\dots),\quad \forall v_j\in V_j,$$
where $(t_{ij})$ is an infinite matrix with entries $t_{ij}\in
\Hom(V_j, V_i)$. Then $(t_{ij})$ satisfies the property: for any
$s>0$ there exist $C$ and $r>0$  such that
$$\|\sum_i \mu_i^{-s} t_{ij}\|<C+\mu_j^{-r},\quad \forall j.$$
\item Conversely any matrix $(t_{ij})$ with entries $t_{ij}\in
\Hom(V_j, V_i)$ satisfying the above property represents a
continuous operator $T:W^{\infty}\to W^{\infty}$.
\end{enumerate}
\end{lem}
\begin{proof}\begin{enumerate}
\item For any $v\in W^{\infty}$, we see that as $n\to\infty$, $\sum_{j=0}^n v_j\to
v$ in $W^s$, therefore $\sum_{j=0}^n v_j\to v$ in $W^{\infty}$.
Because $T$ is continuous, it follows that $T(\sum_{j=0}^n
v_j)=(\sum_{j=0}^n t_{1j} v_j, \sum_{j=0}^n t_{2j} v_j,\dots) \to
Tv$, hence $(Tv)_i = \sum_j t_{ij} v_j$. Suppose the claim is not
true, then there must exist $s>0$, such that for any $C$ and $n$,
there is $j(n,C)$ satisfying
$$\|\sum_i \mu_i^{-s} t_{ij(n,C)}\| > C+\mu_{j(n,C)}^{-n}.$$
Thus one may find an increasing sequence $\{j(n)\}$, such that
$$\|\sum_i \mu_i^{-s} t_{ij(n)}\| > \mu_{j(n)}^{-n}.$$
For each $j$, pick $u_j\in V_j$ so that $\|u_{j(n)}\|=\|\sum_i
\mu_i^{-s} t_{ij(n)}\|^{-1}<\mu_{j(n)}^n$ and $\|\sum_i \mu_i^{-s}
t_{ij(n)}u_{j(n)}\|=1$, while $u_j=0$ if $j\neq j(n), \forall n$.
Then, because of the finite spectral dimension, $u=\sum_j u_j\in
W^{\infty}$, but $T(\sum_{j=0}^nu_j)$ does not converge in $W^{2s}$
as $n\to\infty$, and this yields a contradiction.
\item Suppose the matrix $(t_{ij})$ has that property. Define $T:W^{\infty}\to W^{\infty}$ by
$$(Tv)_i=\sum_j t_{ij}v_j,\quad \forall v\in W^{\infty}.$$
For any sequence $u(n)\in W^{\infty}$, we now prove that if $u(n)\to
0$ as $n\to \infty$ then $Tu(n)\to 0$. For any $s>0$,
$$\|Tu(n)\|_{2s}=\|\sum_i \mu_i^{-s} \sum_j t_{ij} u_j(n)\| \le \sum_j \|\sum_i \mu_i^{-s} t_{ij} u_j(n)\|$$
$$\le\sum_j(C+\mu_j^{-r})\|u_j(n)\| = C\|u(n)\|_{0,1}+\|u(n)\|_{r,1}.$$
Since $u(n)\to 0$ in $W^{\infty}$ implies $\|u(n)\|_{s'}\to 0$ for
all $s'$, using Lemma \ref{estimate}, it follows that
$\|Tu(n)\|_{2s}\to 0$ for all $s$. So this implies $Tu(n)\to 0$, and
therefore $T:W^{\infty}\to W^{\infty}$ is a continuous operator.
\end{enumerate}\end{proof}

For any pre-Hilbert space $\mathcal H$, we define the $*$-algebra
$B(\mathcal H)=\{T\in B(\HH)\mid T(\mathcal H)\subset \mathcal H,
T^*(\mathcal H)\subset \mathcal H\}$.
\begin{defn}\label{unital_sptr}A triple $(\A, \mathcal H, D)$ is said to be a unital \definition{spectral
triple} if it is given by a unital pre-$C^*$-algebra $\A$, a
pre-Hilbert space $\mathcal H$ with a norm-continuous unital
$*$-representation $\A\to B(\mathcal H)$, and a self-adjoint
operator $D$ on $\HH$ called {\em Dirac operator}, with the
following properties:
\begin{enumerate}
\item $D$ has compact resolvent,
\item  $W^{\infty}(D)= \mathcal H$,
\item under the representation of $\A$, the commutator $[D, a]:\mathcal H \to\mathcal H$ is norm-bounded for each $a\in
\A$.
\end{enumerate}
Besides, we always assume that $\A$ is a locally convex topological
$*$-algebra with a topology finer than the norm topology of $\A$,
and the representation $\A\times \mathcal H \to \mathcal H$ is
jointly continuous with respect to the locally convex topology of
$\A$ and the Fr\'echet topology of $\mathcal H$.
\end{defn}

Note that if $(\A, \mathcal H, D)$ is a unital spectral triple, then
$W^1(D)$, the domain of $D$, also forms a left $\A$-module because of
the last condition in the definition.

A spectral triple $(\A, \mathcal H, D)$ is said to be
\definition[spectral triple]{even} if there is a $\Z_2$-grading on $\mathcal H$: $$\mathcal H=\mathcal H_+\oplus \mathcal
H_-,$$ so that the grading operator $$\gamma=\left[\begin{matrix} \id_{\mathcal H_+} &
0\\0 & -\id_{\mathcal H_-} \end{matrix}\right]:\mathcal H\to \mathcal H$$ commutes with all
$a\in \A$ and anti-commutes with $D$. Spectral triples equipped with
no such gradings are said to be \definition[spectral triple]{odd}.

Two spectral triples $(\A_1, \mathcal H_1, D_1)$ and $(\A_2,
\mathcal H_2, D_2)$ with isomorphic algebras $\A_1\cong \A_2$ are said to be
\definition{unitarily equivalent} if there is a unitary operator
$U:\bar{\mathcal H_1}\to\bar{\mathcal H_2}$ intertwining the two
representations of $\A_i$ and the two Dirac operators $D_i$ in an
obvious way. For the even case the unitary operator $U$ also needs to be grade
preserving.
\end{section} 
\begin{section}{Morita equivalence of spectral triples}\label{morita}
In this section we introduce the notion of Morita equivalence of spectral triples. Let $\A$ be a pre-$C^*$-algebra. Recall that a right $\A$-module
$\mathcal S$ is called a {\em pre-Hilbert} (or {\em Hermitian}) {\em
right $\A$-module} if there is an $\A$-valued inner product
$(\cdot,\cdot): \mathcal S \times \mathcal S \to \A$, such that for
all $x,y\in \mathcal S$, $a\in \A$,
\begin{enumerate}
\item $(x,y)=(y,x)^*$,
\item $(x,ya)=(x,y)a$,
\item $(x,x)\geq 0$, and $(x,x)=0$ only if $x=0$.
\end{enumerate}
The norm on $\mathcal S$ is given by $\|x\|=\|(x,x)\|^{1\over 2}$,
and its norm completion $\bar{\mathcal S}$ (which is automatically a
pre-Hilbert right $\bar{\A}$-module) is called a {\em Hilbert right
$\bar{\A}$-module}. Hilbert left modules can be defined in the same
manner. In particular, every Hilbert space is a Hilbert $\C$-module.

If $\mathcal S$ is a pre-Hilbert right $\A$-module, then its
\definition{conjugate space}
$$\mathcal S^*=\{f_x:=(x,\cdot)\mid x\in\mathcal S\}$$
is a pre-Hilbert left $\A$-module, and for all $x,y\in \mathcal S$,
$a\in \A$,
$$af_x :=a(x,\cdot)=f_{x(a^*)},\quad (f_x,f_y):=(y,x),\quad (f_x,a f_y)=a (f_x,f_y).$$

If $\mathcal S$ is a pre-Hilbert right $\A$-module,
$B_{\A}(\bar{\mathcal S)}$ denotes the $*$-algebra of all module
homomorphisms $T:\bar{\mathcal S} \to \bar{\mathcal S}$ for which
there is an adjoint module homomorphism $T^* :\bar{\mathcal S} \to
\bar{\mathcal S}$ with $(Tx,y)=(x, T^*y)$ for all $x,y\in
\bar{\mathcal S}$. Define
$$B_{\A}(\mathcal S)=\{T\in B_{\A}(\bar{\mathcal S})\mid
T(\mathcal S)\subset \mathcal S, T^*(\mathcal S)\subset \mathcal
S\}.$$ In particular, if $\A=\C$, then $B_{\C}(\mathcal
S)$=$B(\mathcal S)$. If $\A$ is unital and $\mathcal S$ is unital
and finitely generated, then $B_{\A}(\mathcal S)$ in fact consists
of all $\A$-endomorphisms of $S$, i.e., $B_{\A}(\mathcal
S)=\End_{\A}(\mathcal S)$.

Suppose $\A$ and $\B$ are unital pre-$C^*$-algebras. Let $\E$ be a
unital finitely generated projective right $\A$-module, then as a
summand of a free module, one can endow $\E$ with a pre-Hilbert
module structure (which is unique up to unitary $\A$-isomorphism).
Suppose $\B$ acts on $\E$ on the left, and the representation $\B\to
B_{\A}(\E)$ is unital, $*$-preserving and norm-continuous. We also
assume that $\E$ is endowed with the topology induced from the
locally convex topology of $\A$, that $\B$ is a locally convex
topological $*$-algebra with a topology finer than the norm
topology, and that the representation $\B\times \E\to \E$ is jointly
continuous with respect to the locally convex topologies on $\B$ and
$\E$. We call such a $\B$-$\A$-bimodule with the above
structure a
\definition{finite Kasparov $\B$-$\A$-module}, then we introduce the
following definition.
\begin{defn}\label{conn}Suppose $\sigma=(\A,\mathcal H,D)$ is a unital spectral triple. A
\definition{$\sigma$-connection} on a finite Kasparov $\B$-$\A$-module $\E$ is a linear mapping
$$\nabla: \E \to \E \otimes_{\A} B(\mathcal H)$$
with the following properties:
\begin{enumerate}
\item $\nabla(\xi a)= (\nabla \xi)a + \xi\otimes_{\A}[D, a]$,
\item $(\xi,\nabla\varepsilon)-(\nabla\xi, \varepsilon)=[D,
(\xi,\varepsilon)]$,
\end{enumerate}
for all $a\in \A$, and $\xi, \varepsilon\in \E$.
\end{defn}
Note that by the notation $(\nabla\xi, \varepsilon)$, which is a bit
ambiguous, we actually mean $(\varepsilon,\nabla\xi)^*\in B(\mathcal
H)$.

Since $\E$ is finitely generated, it follows that for each $b\in
\B$, the commutator $[\nabla, b]$ corresponds to an element in $B(\E
\otimes_{\A} \mathcal H)$. It also follows that any two
$\sigma$-connections on $\E$ differ by a Hermitian element in $B(\E
\otimes_{\A} \mathcal H)$.

From the above data $(\A, \mathcal H, D, \B, \E, \nabla)$, one can
construct a new spectral triple $\sigma^{\E}=(\B, \mathcal H^{\E},
D^{\E})$ for $\B$ (cf. Connes \cite[\S VI.3]{Connes}). The
pre-Hilbert space $\mathcal H^{\E}$ is $\E\otimes_{\A}\mathcal H$
with inner product given by
$$<\xi_1\otimes_{\A}h_1, \xi_2\otimes_{\A}h_2> = <h_1,
(\xi_1,\xi_2) h_2>, \quad \forall \xi_i\in \E, h_i\in \mathcal H.$$
The Dirac operator $D^{\E}$ on $\overline{\mathcal H^{\E}}$ is given
by
$$D^{\E}(\xi\otimes_{\A}h)=\xi\otimes_{\A} Dh +
(\nabla\xi)h,\quad \forall \xi\in\E, h\in W^1(D).$$ It is easy to
see that the commutator $[D^{\E}, b]=[\nabla, b]\in B(\E
\otimes_{\A} \mathcal H)$ is bounded. To verify that $D^{\E}$ is
self-adjoint with domain $\E\otimes_{\A}W^1(D)$, that
$W^{\infty}(D^{\E})=\mathcal H^{\E}$, and that $D^{\E}$ has compact
resolvent, it is adequate to check choosing one particular
$\sigma$-connection on $\E$, because bounded perturbations do not
affect these conclusions.

Recall that a \definition{universal connection} on a pre-Hilbert
right $\A$-module $\mathcal S$ is a linear mapping $\nabla: \mathcal
S \to \mathcal S\otimes_{\A}\Omega_u^1(\A)$ that satisfies the
Leibniz rule
$$\nabla(s a)=(\nabla s)a+\delta_u a,$$ and $\nabla$ is said to
be Hermitian if
$$(s,\nabla\varepsilon)-(\nabla s,\varepsilon)=\delta_u
(s,\varepsilon).$$ Here $\Omega_u^1(\A)$ is the space of universal
$1$-forms of $\A$, and involutions $(\delta_u a)^*$ are set to be
$-\delta_u (a^*)$. Cuntz and Quillen \cite{Cuntz} showed that only
projective modules admit universal connections. Given a universal
Hermitian connection $\nabla$ on a finite Kasparov module $\E$, by sending
$1$-forms $\delta_u a$ to $[D, a]$, one can associate with the
universal connection $\nabla$ a $\sigma$-connection $\nabla_D$ on $\E$ for any spectral
triple $\sigma=(\A,\mathcal H, D)$.

Let $p\in M_n(\A)$ be a projection (i.e., a self-adjoint idempotent
$n\times n$ matrix). Now we consider the right $\A$-module $p\A^n$.
There is a canonical universal connection on $p\A^n$ which is given
by the matrix $p \diag\{\delta_u,\dots, \delta_u\}$ or simply  $\nabla(p\boldsymbol a)= p\delta_u (p\boldsymbol a), \forall
\boldsymbol a\in \A^n$, and this connection is Hermitian.
\begin{defn}A universal connection $\nabla$ on a pre-Hilbert right $\A$-module $\mathcal S$
is said to be
\definition[universal connection]{projectional} if there is a
unitary $\A$-isomorphism $\phi: \mathcal S\to p\A^n$ for some $n$
and some projection $p$ such that $\nabla = \phi^{-1}\circ
p\delta_u \circ\phi$.
\end{defn}

A projectional universal connection on each finite projective
pre-Hilbert $\A$-module is unique up to unitary $\A$-isomorphism. If
$\E$ is a finite Kasparov $\B$-$\A$-module admitting a universal
connection $\nabla^{\E}$ and $\F$ is a finite Kasparov $\mathcal
C$-$\B$-module admitting a universal connection $\nabla^{\F}$, then
there is a twisted universal connection
$\nabla^{\F}\circ\nabla^{\E}$ defined in an obvious way on the
finite Kasparov $\mathcal C$-$\A$-module $\F\otimes_{\B}\E$.
Furthermore if $\nabla^{\E}$ and $\nabla^{\F}$ are projectional then
so is $\nabla^{\F}\circ\nabla^{\E}$. This gives rise to a category --
\definition{the category of noncommutative differential spaces}
$\bf{NDG}$ defined as follows. Objects in $\bf{NDG}$ consist of all unital pre-$C^*$ algebras. If $\A$ is a unital pre-$C^*$ algebra, we denote by $X_{\A}$ the corresponding object in $\bf{NDG}$.  Morphisms from $X_{\A}$ to $X_{\B}$
in $\bf{NDG}$ are isomorphism classes of finite Kasparov
$\B$-$\A$-modules with projectional universal connections. If $\E$ is an finite Kasparov $\B$-$\A$-module with a projective universal connection $\nabla$, we denote by $(\E,\nabla)$ the corresponding morphism in $\Hom_{\mathbf{NDG}}(X_{\A}, X_{\B})$.

Denote by $\rm{Sptr}^0(\A)$ the set of unitary equivalence classes of even spectral
triples for $\A$. The set $\rm{Sptr}^0(\A)$ has an abelian monoid structure, the binary operation of which is the direct sum operation for spectral triples with algebra $\A$.
Thus $\rm{Sptr}^0$ yields a functor from $\mathbf{NDG}$ to $\mathbf{AbMonoid}$, the category of abelian monoids, given by
$$X_{\A}\mapsto \rm{Sptr}^0(\A) \text{ and } \mathrm{Sptr}^0((\E,\nabla)): \sigma\mapsto \sigma^{\E}, \forall \sigma\in  \mathrm{Sptr}^0(\A).$$

\begin{remark} It is not
difficult to extend morphisms of $\bf{NDG}$ to graded modules with
super-connections in the sense of Quillen \cite{superconnection};
however, in this paper except Section \ref{Kasparov}, we only focus on finite Kasparov modules with
trivial gradings.
Baaj-Julg \cite{Baaj} established the theory of unbounded Kasparov
modules and showed that every element in the bivariant KK-theory can
be represented by an unbounded Kasparov module. It would be nice if
morphisms of $\bf{NDG}$ could be enlarged to unbounded (even)
Kasparov modules $(E, \mathscr D,\Gamma)$ with a grading $\Gamma$ and an appropriate universal
connection $\nabla$, and thereby $D^{E,\mathscr
D,\Gamma}(\xi\otimes_{\A}h)=\mathscr D\xi\otimes_{\A} h + \Gamma
\xi\otimes_{\A} Dh + \Gamma(\nabla\xi)h$. The notion of connection
for bounded Kasparov modules introduced by Connes-Skandalis was
well-known, whereas theory of connection for unbounded Kasparov modules has been developed in a recent work by Mesland \cite{Bram}.
\end{remark}

\begin{defn}Two unital spectral triples $\sigma_1=(\A_1, \mathcal H_1, D_1)$ and $\sigma_2=(\A_2, \mathcal H_2,$  $D_2)$
are said to be \definition{Morita equivalent} if $\A_1$ and $\A_2$
are Morita equivalent as algebras and there is a finite Kasparov
$\A_2$-$\A_1$-module $\E$ with a $\sigma_1$-connection, such that
$\E$ is an equivalence bimodule and $\sigma_2$ is unitarily
equivalent to $\sigma_1^{\E}$.
\end{defn}
Notice that the definition of Morita equivalence for spectral triples here, using $\sigma$-connections rather than the usual universal connections, will give rise to a symmetric relation:
\begin{thm} The Morita equivalence between spectral triples is an
equivalence relation.
\end{thm}
\begin{proof}Reflexivity: $(\A, \mathcal H, D)$ is Morita equivalent to itself via
the trivial connection $\nabla a= [D, a]$ on $\A$.\\
Transitivity: It is straight forward by definition. \\
Symmetry: Suppose $(\B, \mathcal H^{\E}, D^{\E})$ is Morita
equivalent to $\sigma=(\A, \mathcal H, D)$ via a $\sigma$-connection
$\nabla$ on $\E$. Let $\F=\E^*$. Define $\nabla^{\F}:\F \to
\F\otimes_{\B}B(\mathcal H^{\E})$ by
$$(\nabla^{\F} f) (\xi\otimes_{\A}h) = -f(\nabla\xi)\otimes_{\A}h+[D,f\xi]h,
\quad \forall f\in \F, \xi\in\E,h\in \mathcal H.$$ Here we use a map
$\mathcal H^{\E}\to \mathcal H$ to represent an element in
$\F\otimes_{\B}B(\mathcal H^{\E})$. One can verify that $\nabla^{\F}$ is
a $(\B,\mathcal H^{\E},D^{\E})$-connection. Now we check
$(D^{\E})^{\F}=D$ as follows,
\begin{eqnarray*}(D^{\E})^{\F}(f\otimes_{\B}\xi\otimes_{\A}h)
&=& f \otimes_{\B} D^{\E}(\xi\otimes_{\A}h) + (\nabla^{\F}
f)(\xi\otimes_{\A}h) \\
&=& f\otimes_{\B}\xi Dh + f(\nabla \xi)h + (\nabla^{\F}
f)(\xi\otimes_{\A}h) \\
&=& f\xi Dh + f(\nabla \xi)h - f(\nabla \xi)h + [D, f\xi]h \\
&=& D(f\xi h).
\end{eqnarray*}
(An alternative way to prove the symmetry property is using bounded
perturbations of the $\sigma$-connection and
$\sigma^{\E}$-connection that are associated to projectional universal
connections.)

\noindent In conclusion, Morita equivalence of spectral triples is
an equivalence relation.
\end{proof}

\begin{remark}
There are at least three possible spaces of ``connections'' for spectral triples:
\begin{enumerate}
\item The space of $\sigma$-connections.
\item The space of {\em 1-form connections}, i.e., connections induced from universal connections, which is a subspace of type i).
\item The space of connections induced from projectional universal connections, which is a subspace of type ii).
\end{enumerate}
There are several reasons not restricting ourselves to the 1-form connections when trying to define Morita relation between spectral triples:
\begin{enumerate}
\item The Morita relation defined using 1-form connections is not symmetric.
\item For the $Z_2$-graded cases (see Section \ref{Kasparov}), one need to use super-connections which are in general high order form connections \cite{GetzlerBook}.
\item The Morita relation defined using 1-form connections does not include all bounded perturbations of Dirac operators.
\end{enumerate}
As a special case, $\A$ is Morita equivalent to itself via $\E=\A$.
Using universal connections on $\A$, one can construct new spectral
triples which are called inner fluctuations of spectral triples
\cite{Chamse}. {\em Morita equivalence of spectral triples with common algebra $\A$ and pre-Hilbert space
$\mathcal H$ is exactly bounded perturbation of Dirac operators.} In fact, $(\A,\mathcal H,D)$ is Morita equivalent to $(\A,\mathcal H,D+B)$ for any bounded self-adjoint operator $B$ (or bounded odd self-adjoint operator $B$ in the even case) through the bimodule $\A$ and the $(\A,\mathcal H,D)$-connection $\nabla$ given by $\nabla(1):=B$. Conversely,  $(\A,\mathcal H,D+B)$ is Morita equivalent to $(\A,\mathcal H,D)$ through the $(\A,\mathcal H,D+B)$-connection $\nabla^*$ given by $\nabla^*(1)=-B$. This is not the case for inner fluctuation of spectral
triples, as the latter is generally not an equivalence relation when
$\A$ is noncommutative.  In general, if we confine ourselves to
equivalence bimodules with universal connections, the symmetry
property in the above proof does not hold. However, if we restrict the universal connections
on equivalence bimodules to the projectional ones, the symmetry property
holds again. Unfortunately, Levi-Civita connections are not projectional connections.
\end{remark}

Suppose $(\E,\nabla)$ is a morphism in $\Hom_{\bf{NDG}}(X_{\A},X_{\B})$, then there is a free right $\A$-module $\A^m$ and a projection $p$ in $\M_m(\A)$ such that $\E\cong p\A^m$. Let $\{e_1,...,e_m\}$ be the standard generators of $\A^m$. For each $b$ in $\B$, one can find a matrix $\alpha(b)$ in $p\M_m (\A)p$, such that
$bpe_i=\sum\limits_j e_j\alpha_{ji}(b)$.

The K-theory, Hochschild (co)homology and (periodic) cyclic (co)homology are all functors on the category $\bf{NDG}$. For instance, suppose $(b_0,...,b_n)$ is a Hochschild $n$-cycle representing an element $\Phi\in HH_*(\B)$, then $\E(\Phi)\in HH_*(\A)$ is the Hochschild $n$-cycle given by the Dennis trace map
$$\tr(\alpha(b_0)\otimes\cdots\otimes \alpha(b_n))=\sum_{i_0,...,i_n}(\alpha_{i_0i_1}(b_0),\alpha_{i_1i_2}(b_1),...,\alpha_{i_ni_0}(b_n)).$$
Note that $\alpha$ depends on the choice of the isomorphism $\E\cong p\A^m$; however, it does induce the well-defined morphisms $\E:HH_*(B)\to HH_*(\A)$.

Furthermore, the Connes-Chern characters \cite{NDG} are natural transformations from the K-theory functor $X_{\A}\mapsto K_0(\A)$ to the periodic cyclic homology functor $X_{\A}\mapsto HP_0(P)$, and from functors $X_{\A}\mapsto \rm{Sptr}^0(\A)$ and $X_{\A}\mapsto K^0(\bar{\A})$ to the periodic cyclic cohomology functor $X_{\A}\mapsto HP^0(\A)$. The naturalness of Connes-Chern characters is illustrated in the following commutative diagrams
$$
\CD
K_0(\B) @>\E>> K_0(\A) \\
@V\ch VV  @V\ch VV  \\
HP_0(\B) @>\E>> HP_0(\A),
\endCD
\quad\qquad
\CD
\rm{Sptr}^0(\A) @>(\E,\nabla)>> \rm{Sptr}^0(\B) \\
@V\ch VV  @V\ch VV  \\
HP^0(\A) @>\E>> HP_0(\B).
\endCD
$$

\begin{prop}The following diagrams
$$
\CD
K_0(\A) &\quad \times\quad & \rm{Sptr}^0(\A) @>\ind>> \Z \\
@AA \E A  @VV(\E,\nabla) V   @| \\
K_0(\B) &\quad \times\quad & \rm{Sptr}^0(\B) @>\ind>> \Z,
\endCD
\quad\qquad
\CD
HP_0(\A) &\quad \times\quad & HP^0(\A) @>>> \C \\
@AA \E A  @VV\E V   @| \\
HP_0(\B) &\quad \times\quad & HP^0(\B) @>>> \C,
\endCD
$$and
$$
\CD
K_0(\A) &\quad \times\quad & \rm{Sptr}^0(\A) @>\ind>> \Z \\
@VV \ch V  @VV\ch V   @VVV \\
HP_0(\A) &\quad \times\quad & HP^0(\A) @>>> \C,
\endCD
$$
commute.
\end{prop}
\end{section} 
\section[Gluing local spin structures via Morita equivalence]{Gluing local spin structures on Riemannian manifolds via Morita equivalence}\label{glue}

We can apply the above theory to spectral triples on
Riemannian manifolds. By gluing local pieces
of spectral triples via Morita equivalence, we construct a so called
projective spectral triple, the Dirac operator of which was defined
in a formal sense by Mathai-Melrose-Singer \cite{Mathai}.

Let $X$ be a closed oriented Riemannian manifold of dimension $n$.
Suppose $X$ is spin or spin$^c$. Let $\Cl_n$ denote the complex
Clifford algebra of $\R^n$. In this paper we use the following
convention for the definition of Clifford algebras
$$\Cl_n:=<u\in \mathbb R^n| uv+vu=-2(u,v), \forall u,v\in\mathbb R^n>.$$
Decomposed by the parity of the degree, $\Cl_n=\Cl_n^0 \oplus \Cl_n^1$.
Write
\begin{equation}B_n=\left\{
\begin{array}{c} \Cl_n \cong M_{2^m}(\C) \\
\Cl_n^0 \cong M_{2^m}(\C)
\end{array}\right. \text{and}\quad
B_x=\left\{
\begin{array}{l} \Cl(T_x^*X)  \\
\Cl^0(T_x^*X)
\end{array}\right.
\begin{array}{l} \text{when}\ n=2m\ \text{is even}  \\
\text{when}\ n=2m+1\ \text{is odd.}
\end{array}
\end{equation}
Denote by $\Cl(X)$ and $B(X)$ the vector bundles over $X$
whose fibers at a point $x\in X$ are $\Cl(T_x^*X)$ and $B_x$. Let $S_n=\C^{2^m}$ be the standard spinor vector
space, and we fix a specific isomorphism
$$c : B_n \rightarrow \End_{\C}(S_n).$$
Let $\omega_n=i^{[{n+1\over 2}]}e_1\cdots e_n \in \Cl_n$.
When $n$ is odd, $B_n = \omega_n\Cl_n^1$. This indicates a
homomorphism $c:\Cl_n \rightarrow \End_{\C}(S_n).$ When $n$ is even,
$S_n=S_{n+}\oplus S_{n-}$ is graded by the eigenspaces of
$\omega_n$, which are invariant under the action of $\Spin(n)$.

Let $P_{\rm{Fr}}(X)$ and $P_{\Spin(^c)}(X)$ denote the
orthonormal oriented frame bundle and the principal
$\Spin(n)$ or $\Spin^c(n)$-bundle over $X$ respectively. Then
$$\Cl(X)=P_{\rm{Fr}}(X) \times_{\SO(n)} \Cl_n = P_{\Spin(^c)}(X) \times_{\Ad} \Cl_n.$$
The spinor bundle over $X$ is the associated $\Spin({^c})(n)$-bundle
$$S_X=P_{\Spin({^c})}(X) \times_{c} S_n.$$
The Clifford algebra bundle acts naturally on the spinor bundle, $c:
\Cl(X)\times_X S_X \rightarrow S_X$, which is given by
$$(p,\xi)\times (p,s) \mapsto (p, c(\xi)s), \quad \forall p\in P_{\Spin(^c)}(X), \xi\in
\Cl_n, s\in S_n.$$
When $n$ is even, $\omega_n$ induces a grading operator $\omega$ on $S_X=S_{X+}\oplus S_{X-}$.

Denote by $\D$ the Dirac operator on $S_X$. Let
$E$ be a Hermitian vector bundle over $M$ with a Hermitian
connection $\nabla^E$. Let $\A=\smooth(X),\  \B=\Gamma(B(X))$,
$\E=\Gamma(E)$. Then the well-known spin spectral triple
\begin{equation}
(\A, \mathcal H, D)=(\smooth(X),\Gamma(S_X),\D)
\end{equation}
is Morita equivalent to the following spectral triple with a
noncommutative algebra
\begin{equation}(\B, \mathcal H^{\E}, D^{\E})=(\Gamma(\End(E)), \Gamma(E\otimes S_X),
\D^E),
\end{equation}
via the finite Kasparov module $\E$ with $(\A,\mathcal H,
D)$-connection associated with $\nabla^E$. Here $\D^E$ denotes the
twisted Dirac operator on the vector bundle $E\otimes S_X$.

Now interesting things happen when a manifold has no spin$^c$
structure: The spinor bundle does not exist, and neither does the
spin spectral triple. However, as constructed later in this section, for any closed
oriented Riemannian manifold, not necessarily spin$^c$, there is a canonical noncommutative spectral triple:

\begin{defn}The \definition{projective spectral triple} of a closed oriented Riemannian manifold $M$ is defined to be
\begin{equation}(\A_{W_3},\mathcal H_{W_3},D_{W_3}):=(\Gamma(B(M)),
(1+*)\Omega(M),(d-d^*)(-1)^{\deg}),
\end{equation}
if $M$ is odd dimensional, and
\begin{equation}(\A_{W_3},\mathcal H_{W_3},D_{W_3},\gamma_{W_3}):=(\Gamma(B(M)), \Omega(M),(d-d^*)(-1)^{\deg},*(-1)^{{\deg(\deg+1)\over 2}-{n\over 4}}),
\end{equation}
if $M$ is even dimensional.
\end{defn}

We can consider that projective spectral triples are obtained by gluing local spin spectral triples in the following way:

Let $M$ be a closed oriented Riemannian manifold of dimension $n$,
not necessarily spin$^c$. Let $\{U_i\}$ be a good covering of $M$.
Then on each local piece we have the principal $\Spin(n)$-bundle
$P_i=P_{\Spin}(U_i)$, the associated spinor bundle $S_i=S_{U_i}$,
the spin connection $\nabla_i$, and the Dirac operator
$\slashed{D}_i$. Over each intersection $U_{ij}=U_i\cap U_j\neq
\emptyset$, there is up to $\pm1\in \Spin(n)$ a unique homeomorphism
of principal bundles
$$\alpha_{ij}: P_i|_{U_{ij}}\rightarrow P_j|_{U_{ij}}, \text{ such that }
\alpha_{ij}(p_i g)=\alpha_{ij}(p_i)g,\ \forall p_i\in P_i|_{U_{ij}},\
g\in \Spin(n).$$ This homeomorphism induces up to $\pm1$ a Clifford
module homomorphism $\alpha_{ij}:S_i|_{U_{ij}}\rightarrow
S_j|_{U_{ij}}$,
$$\alpha_{ij}(p,s)=(\alpha_{ij}(p),s),\quad \forall p\in P_i|_{U_{ij}}, s\in S_n.$$
For each triple overlap $U_{ijk}=U_i\cap U_j\cap U_k\neq \emptyset$,
write
$$\sigma_{ijk}=\alpha_{ki}\circ\alpha_{jk}\circ\alpha_{ij}|_{U_{ijk}},$$ then $\{\sigma_{ijk}\}$ represents a
\v Cech cocycle in $\check{H}^2(M,\Z_2)$, and the corresponding
singular class $w_2(M)\in H^2(M,\Z_2)$ is the second Stiefel-Whitney
class. Let $L_{ij}=\Hom_{\Cl(U_{ij})}(S_i|_{U_{ij}}, S_j|_{U_{ij}})$
denote the line bundle of the Clifford module homomorphisms of $S_i$
and $S_j$ over $U_{ij}$, then $\alpha_{ij}$ is the canonical section
of $L_{ij}$. $\{L_{ij}\}$ forms a gerbe of line bundles over $M$
characterized by the Dixmier-Douady class $\delta(B(M))=W_3(M)$ (the
third integral Stiefel-Whitney class). We refer \cite{Bouw} for
(twisted) K-theory of bundle gerbes.

If $M$ is spin, then for each $U_{ij}$ there is $\beta_{ij}\in
\Z_2$ such that $\{\beta_{ij}\alpha_{ij}\}$ satisfies the cocycle
condition. That means the local spinor bundles $S_{ij}$ can be glued
together as a global spinor bundle via the Clifford module
isomorphisms $\beta_{ij}\alpha_{ij}$. The difference between two
different choices of such collection $\{\beta_{ij}\}$ is a cocycle in
$\check{H}^1(M,\Z_2)$ which parameterizes distinct spin structures on
$M$.

If $M$ is spin$^c$, then for each $U_{ij}$ there is $\beta_{ij}\in
\smooth(U_{ij},U(1))$ such that $\{\beta_{ij}\alpha_{ij}\}$
satisfies the cocycle condition. That means that the local spinor bundles
$S_{ij}$ can be glued together as a global spinor bundle via the
Clifford module isomorphisms $\beta_{ij}\alpha_{ij}$. The collection $\{\beta_{ij}^2\}$
also satisfies the cocycle condition, and the \v Cech cocycle
$\{\beta_{ij}^2\}\in \check{H}^1(M,U(1))$ corresponds to the
canonical line bundle $\sL$ of a spin$^c$ structure. The difference
between two different choices of such collection $\{\beta_{ij}\}$ is a cocycle in
$\check{H}^1(M,U(1))$ which parameterizes distinct spinor bundles on
$M$.

Denote by $\smooth(\bar{U_i})$ the space of smooth functions on
$U_i$ that can be extended to a smooth function on a small open
neighborhood $V_i$ of $\bar{U_i}$, and likewise denote by $\Gamma(\bar{U_i},\cdot)$ for extendable smooth sections. Take $\E_i=\Gamma(\bar{U_i},S_i)$, then
the ``local spin spectral triple''

\begin{equation}\label{local_sptr1}
(\A_i, \mathcal H_i,
D_i)=(\smooth(\bar{U_i}),\Gamma(\bar{U_i},S_i),\D_i)
\end{equation}
is Morita equivalent to the local spectral triple
\begin{equation}\label{local_sptr2}(\B_i, \mathcal H_i^{\E_i}, D_i^{\E_i})=(\Gamma(\bar{U_i},B(U_i)), \Gamma(\bar{U_i}, S_i\otimes S_i),
\D_i^{S_i}).
\end{equation}

\begin{remark}We use here the triple (\ref{local_sptr1}) to formulate the spin structure of  the open subset $U_i$, but it is by definition not a spectral triple, for the compact resolvent condition fails. However, it can naturally act on the relative K-theory for the pair of spaces $(V_i,U_i)$ or $(Y,\iota(U_i))$ to get an index, where $Y$ is any compact Riemannian spin manifold that admits an isometric embedding $\iota: V_i \to Y$. By the excision property, the choice of $V_i$ is irrelevant. In this sense the triple (\ref{local_sptr1}) represents a relative K-cycle. One may also think of the standard treatment using the nonunital spectral triple $(\smooth_c(U_i), L^2(U_i, S_i), \D_i)$ with the algebra of smooth functions with compact support. See Gayral-Gracia-Bond\'ia-Iochum-Sch\"ucker-V\'arilly \cite{Moyal} for a set of axioms for nonunital spectral triples which is proposed to set up the notion of noncompact noncommutative spin manifolds. This, however, will cause some subtleties when considering Morita equivalence and the smoothness condition.
\end{remark}

Because the collection of maps $\{\alpha_{ij}\otimes\alpha_{ij}\}$
satisfy the cocycle condition, the vector bundles $S_i\otimes S_i$
and Dirac operators $D_i^{\E_i}$ can be glued together to form a
vector bundle $N$ over $M$ and a Dirac operator $D_N$ on $N$, so that
$(\B|_{U_i},\Gamma(M, N)|_{U_i},D_N|_{U_i})$ are unitarily
equivalent to $(\B_i, \mathcal H_i^{\E_i}, D_i^{\E_i})$, where
$\B=\Gamma(M, B(M))$. Thus we succeed to construct a globally well-defined spectral
triple $(\B, \Gamma(N), D_N)$ on $M$.

\begin{prop} The global vector bundle $N$ is isomorphic to $B(M)$,
and $\Gamma(N)$ is isomorphic to $\Gamma(B(M))$ as both
$\Gamma(B(M))$-modules and pre-Hilbert spaces.
\end{prop}
\begin{proof} Let $S_n^*$ be the dual vector space of the standard
spinor vector space $S_n$. We endow $S_n^*$ with a left $B_n$-module
structure by $\gamma^*(b)f_x:=f_{\bar{b}x}=(x,\bar{b}^*\cdot)$, for
all $x\in S_n, b\in B_n$, where $\bar{b}$ is the complex conjugate
of $b$, and $b^*$ is the adjoint of $b\in B_n\cong M_{2^m}(\C)$.
Since $B_n$ is a simple algebra, up to a scalar there is a unique
$B_n$-module isomorphism from $S_n$ to $S_n^*$. We fix one specific
unitary $B_n$-module isomorphism $T_n: S_n\rightarrow S_n^*$. Then
$T_n$ induces a Clifford module isomorphism of bundles $T:S_i
\rightarrow S_i^*$, given by $(p,s)\mapsto (p, T_n s)$, for all
$p\in P_i, s\in S_n$. Let $S_i^*=P_i \times_{\gamma^*} S_n^*$, and
let $\alpha_{ij}^*:S_i^*\rightarrow S_j^*$ denote the Clifford
module isomorphism given by $\alpha_{ij}^*(p,f))=(\alpha_{ij}p,f)$,
for all $f\in S_i^*$. The mappings $T$ on $U_i$ and $U_j$ are
compatible, namely the diagram below commutes on $U_{ij}$.
$$\xymatrix{S_i
   \ar[d]_{\alpha_{ij}} \ar[r]^{T} & S_i^* \ar[d]^{\alpha_{ij}^*} \\
  S_j \ar[r]^{T} & S_j^*  }$$
 Then one can glue the bundles $S_i^*\otimes S_i
\cong \End_{\C}(S_i^*)$ together as a global bundle $B(M)$ via the
maps $\alpha_{ij}^* \otimes \alpha_{ij}$, and $T$ induces an
isomorphism from $N$ to $B(M)$. Then it is easy to verify the
proposition.
\end{proof}

\begin{cor}\label{projectivesptr}When $M$ is odd dimensional, one can take the bundle
$N=(1+*)\bigwedge^*(T^*_{\C}M)$, then the spectral triple $(\B, \Gamma(N),D_N)$ is the projective spectral triple
\begin{equation}(\A_{W_3},\mathcal H_{W_3},D_{W_3})=(\B,
(1+*)\Omega(M),(d-d^*)(-1)^{\deg});\end{equation} and when $M$ is
even dimensional, take $N=\bigwedge^*(T^*_{\C}M)$, then the spectral triple $(\B, \Gamma(N),D_N)$ with the grading on $\Gamma(N)$ obtained from the grading on $S_i$ is the even projective spectral triple
\begin{equation}(\A_{W_3},\mathcal H_{W_3},D_{W_3},\gamma_{W_3})=(\B, \Omega(M),(d-d^*)(-1)^{\deg},*(-1)^{\deg(\deg+1)/2-n/4}),\end{equation}
and its center $(\A, \mathcal H_{W_3},D_{W_3},\gamma_{W_3})$ is
unitarily equivalent to the spectral triple for Hirzebruch
signature. For any $a\in \A$, the commutator $[D_{W_3}, a]$ is just
the right Clifford action of $da$ on $\mathcal H_{W_3}$.
\end{cor}

\begin{thm}If $M$ is spin$^c$, the projective spectral triple
on $M$ is Morita equivalent to the spin spectral triple.
\end{thm}
\begin{proof}Let $S$ be the global spinor bundle over $M$. There exist local half line
bundles $\sL^{1/2}_i$, such that $\sL^{1/2}_i$ are characterized by
$\{\beta_{ij}\}$ as local transition functions, $\sL^{1/2}_i\otimes
\sL^{1/2}_i=\sL|_{U_i}$, and $S|_{U_i}=\sL^{1/2}_i\otimes S_i$.
There also exists a hermitian connection $\nabla_i'$ on $\sL^{1/2}_i$
such that the Spin$^c$-connection
$\nabla|_{U_i}=1\otimes\nabla_i+\nabla_i'\otimes 1$, where $\nabla_i$
is the Spin-connection on $S_i$. Then follows the Morita
equivalence of the spectral triples
$$(\smooth(M),\Gamma(M, S),\D)=(\A,\mathcal H,D) \sim (\A,\mathcal H^{\sL^{-1}},D^{\sL^{-1}})
$$$$\sim (\B,(\mathcal
H^{\sL^{-1}})^S,(D^{\sL^{-1}})^S)=(\A_{W_3},\mathcal
H_{W_3},D_{W_3}).$$
\end{proof}

We see that the projective spectral triple is defined for any closed
oriented Riemannian manifold regardless of whether the manifold is
spin$^c$ or not. The projective spectral triple depends only on the
metric and orientation of $M$ and does not depend on the choice of
the local spinor bundles $S_i$.

\begin{section}{Kasparov's spectral triple}\label{Kasparov}

In Introduction we mentioned that Kasparov's fundamental class, the Dirac element \cite{Kasparov}, is essentially an even spectral triple with a $\Z_2$-graded underlying algebra, and claimed that the projective spectral triple is in fact Morita equivalent to Kasparov's spectral triple. In this section we see in details how these two spectral triples are related by Morita equivalence.

It is very easy to make sense of even spectral triples with $\mathbb Z_2$-graded algebras and Morita equivalence between them. Replacing all the commutator relations for even spectral triples by graded commutator relations, the definition of even spectral triples with $\Z_2$-graded algebras will be the same as that of even spectral triples with ungraded algebras. Similarly, we can generalize $\sigma$-connections to $\sigma$-supper-connections. To be compatible with graded commutator relations, $\sigma$-supper-connections are required to be
odd operators. Two $\mathbb Z_2$-graded algebras $\A_1$ and $\A_2$ are Morita equivalent through a $\Z_2$-graded $\A_2$-$\A_1$-bimodule $\E$ with grading $\gamma_{21}$, if $\A_1$ and $\A_2$ as ungraded algebras are Morita equivalent through $\E$ as an ungraded module and
 $\A_2\cong \End_{\A_1}(\E)$ as $\Z_2$-graded algebras. Suppose $\sigma=(\A_1,\mathcal{H}_1, D_1, \gamma_1)$ is an even spectral triple with $\Z_2$-graded algebra $\A_1$, $\E$ is  a $\Z_2$-graded $\A_2$-$\A_1$-bimodule with grading $\gamma_{21}$, and $\nabla$ is a $\sigma$-super-connection  on $\E$, we define
 $$D_1^{\E}(\xi\otimes_{\A_1}h)=  \gamma_{21} \xi\otimes_{\A_1} D_1h + \gamma_{21}(\nabla\xi)h,\ \forall h\in \mathcal{H}_1, \xi\in \E,$$
 $$\gamma_1^{\E}= \gamma_{21}\otimes_{\A_1}\gamma_1.$$
Then we can straightforwardly extend the notion of Morita equivalence of spectral triples to the case of even spectral triples with $\Z_2$-graded algebras.

Recall that the projective spectral triple for an even dimensional closed Riemannian manifold $M$ is an even spectral triple
$$(\A_1,\mathcal{H}_1, D_1, \gamma_1)=(\B, \Omega(M), (d-d^*)(-1)^{\deg}, *(-1)^{\deg(\deg+1)/2-\dim M/4}),$$
where $\B$ is the algebra of smooth sections of the Clifford bundle of $M$ with trivial grading.

We define {\em a variant of Kasparov's spectral triple} to be the even spectral triple with a $\mathbb Z_2$-graded algebra
$$(\A_2, \mathcal{H}_2, D_2, \gamma_2)=(\B_{\gr}, \Omega(M), \sqrt{-1}(d-d^*), (-1)^{\deg}),$$
where $\B_{\gr}$ is the algebra of smooth sections of the Clifford bundle of $M$ graded by the degree of its elements.

\begin{prop} $(\A_1,\mathcal{H}_1, D_1, \gamma_1)$ and $(\A_2, \mathcal{H}_2, D_2, \gamma_2)$ are Morita equivalent.
\end{prop}
\begin{proof} First note that $(\A_1,\mathcal{H}_1, D_1, \gamma_1)$ is unitarily equivalent to $(\A_1,\mathcal{H}_1, D_1', \gamma_1)$, where $D_1'=\sqrt{-1}D_1 \gamma_1$.
Then the Morita equivalence between $(\A_1,\mathcal{H}_1, D_1, \gamma_1)$ and $(\A_2, \mathcal{H}_2, D_2, \gamma_2)$ will be given by a graded $\A_2$-$\A_1$-bimodule
$\E_{21}$ with grading $\gamma_{21}$, together with a super-connection $\widehat{\nabla}$ on $\E_{21}$, such that
$$(\A_2, \mathcal{H}_2, D_2, \gamma_2)=(\End_{\A_1}(\E_{21}), \E_{21}\otimes_{\A_1}\mathcal{H}_1, D_1'^{\E_{21}}, \gamma_{21}\otimes_{\A_1} \gamma_1).$$
The construction of $(\E_{21},\widehat{\nabla})$ is as follows.

Denote by $c_L$ and $c_R$ the left and right Clifford action of $\B$ on $\Omega(M)$ respectively. $\B$ is a complex algebra generated by $\Gamma(M, T^*_{\R} M)$. For any $\alpha \in \Gamma(M, T^*_{\R} M)$, and $\omega \in \Omega^k(M)$, we have
\begin{equation}\label{AA}c_L(\alpha) \omega= \alpha\wedge \omega -\iota(\alpha) \omega,\end{equation}
\begin{equation}\label{BB}c_R(\alpha) \omega= (-1)^k(\alpha\wedge \omega +\iota(\alpha) \omega),\end{equation}
where $\iota$ is the contraction determined by the Riemannian metric $g$. Note that $c_L(\alpha)c_R(\beta)=c_R(\beta)c_L(\alpha)$ and that
\begin{equation}\label{CC}*(a\wedge \omega)=(-1)^k \iota(\alpha)(*\omega),\end{equation}
\begin{equation}\label{DD}*(\iota(\alpha) \omega)= (-1)^{k+1} \alpha \wedge (*\omega).\end{equation}

Let $\E_{21}:=\Omega(M)$ be the $\B_{\gr}$-$\B$-bimodule with actions of $\B_{\gr}$ and $\B$ being the left and right Clifford action respectively.
Now we consider three different gradings $\gamma_1$, $\gamma_{21}$ and $\gamma_2$ on $\Omega(M)$. We define
$$\gamma_{21}(\omega):=(-1)^{k(k-1)/2-n/4}*\omega,\quad \forall \omega \in \Omega^k(M),$$
where $n=\dim M$ is an even number.
One can easily check that $\gamma_1^2=\gamma_2^2=\gamma_{21}^2=\id.$
By (\ref{AA})(\ref{BB})(\ref{CC})(\ref{DD}) we have
\begin{equation}\label{EE}c_L(\alpha) \gamma_1(\omega)= \gamma_1(c_L(\alpha)\omega),\end{equation}
\begin{equation}\label{FF}c_R(\alpha) \gamma_{21}(\omega)= \gamma_{21}(c_R(\alpha)\omega).\end{equation}
We consider $\Omega(M)\xrightarrow{\cong} \E_{21}\otimes_{\B} \mathcal{H}_1$ via
$\omega\mapsto \omega\otimes_{\B} 1,$
and this is an isomorphism of left $\B$-modules.
By (\ref{EE})(\ref{FF}), $\gamma_{21}\otimes_{\B}\gamma_1$ is well-defined. Hence
$$(\gamma_{21}\otimes_{\B}\gamma_1)\omega = \gamma_{21}(\omega)\otimes_{\B}\gamma_1(1)$$
$$=(-1)^{k(k-1)/2-n/2}(*\omega) \otimes_{\B}(e_1\wedge\cdots \wedge e_n)$$
$$=(-1)^{k(k-1)/2-n/2}(*\omega)\otimes_{\B}c_L(e_1)\cdots c_L(e_n) 1$$
$$=(-1)^{k(k-1)/2-n/2}  c_R(e_n)\cdots c_R(e_1)(*\omega)\otimes_{\B} 1$$
$$=(-1)^{k(k-1)/2-n/2}(-1)^{(n-k)(n-k+1)/2}(**\omega)\otimes_{\B} 1$$
$$=(-1)^k \omega \otimes_{\B} 1=\gamma_2(\omega).$$
Again by (\ref{AA})(\ref{BB})(\ref{CC})(\ref{DD}) we have
$$c_L(\alpha) \gamma_{21}(\omega)= -\gamma_{21}(c_L(\alpha)\omega),$$
so the grading $\gamma_{21}$ on $\E_{21}$ is compatible with the grading on $\B_{\gr}$, and hence $\B_{\gr}= \End_{\B}(\E_{21})$ as $\Z_2$-graded algebras.

Let $\nabla$ be the connection on $\Omega(M)$ induced by the Levi-Civita connection on $T_{\R}^*M$. Note that
if $$\nabla(\omega)=\sum_i \nu_i\otimes\mu_i,$$ for some $\mu_i\in \Omega^k(M), \nu_i\in \Omega^1(M)$, then $d\omega=\sum_i \nu_i\wedge \mu_i$, $d^* \omega=-\sum_i \iota(\nu_i) \mu_i$, therefore
$$(d+d^*)\omega=\sum_i c_L(\nu_i) \mu_i,$$
$$(d-d^*)\omega=(-1)^k \sum_i c_R(\nu_i) \mu_i,$$
namely$$d+d^*=c_L\circ \nabla,$$
$$ D_1=(-1)^k (d-d^*)=c_R \circ \nabla.$$

Define $\widehat{\nabla}: \E_{21}\mapsto \E_{21}\otimes_{\B} B(\mathcal{H}_1)$  by
$$(\widehat{\nabla}(\omega))(h)=\sqrt{-1}\sum_i\gamma_1(\mu_i)\otimes_{\B} c_R(\nu_i)h,$$
that is
$$\widehat{\nabla}(\omega)=\sqrt{-1}\ c_R \circ \nabla(\gamma_1(\omega)).$$
So $\widehat{\nabla}$ is a $(\A_1,\mathcal{H}_1, D_1', \gamma_1)$-connection, and we have
$$D_1'^{\E_{21}}(\omega)=D_1'^{\E_{21}}(\omega \otimes_{\B} 1)=(\gamma_{21}\otimes_{\B}\id)(\widehat{\nabla}(\omega))(1)+\gamma_{21}(\omega) \otimes_{\B}D_1'(1)$$
$$=\sqrt{-1}(\gamma_{21}\otimes_{\B}\id) \sum_i\gamma_1(\mu_i)\otimes_{\B} c_R(\nu_i)1+ 0$$
$$=\sqrt{-1} (-1)^k \sum_i c_R(\nu_i) \mu_i=\sqrt{-1}(d-d^*)\omega.$$
So  $D_2=D_1'^{\E_{21}}$.

Therefore we have the Morita equivalence
$$(\A_2, \mathcal{H}_2, D_2, \gamma_2)=(\End_{\A_1}(\E_{21}), \E_{21}\otimes_{\A_1}\mathcal{H}_1, D_1'^{\E_{21}}, \gamma_{21}\otimes_{\A_1} \gamma_1)$$
$$\sim (\A_1, \mathcal{H}_1, D_1, \gamma_1).$$
\end{proof}

Now let $D_3=d+d^*$ and $\gamma_3=\gamma_2=(-1)^{\deg}$. Note that for even dimensional manifolds, $$*^{-1}\sqrt{-1}(d-d^*)* =(d+d^*)(-1)^{\deg+1/2}=\sqrt{-1}D_3\gamma_3.$$
If we write $\gamma_3=\left[\begin{matrix} \id &
0\\0 & -\id \end{matrix}\right]$ and $D_3=\left[\begin{matrix} 0 & T
\\ T^* & 0 \end{matrix}\right]$, then
$$*^{-1} D_2 *=\sqrt{-1}D_3\gamma_3=\left[\begin{matrix} 1 & 0
\\ 0 & \sqrt{-1} \end{matrix}\right] D_3 \left[\begin{matrix} 1 & 0
\\ 0 & -\sqrt{-1} \end{matrix}\right],$$
so we have
\begin{equation}\label{HH}D_3= \left[\begin{matrix} 1 & 0
\\ 0 & -\sqrt{-1} \end{matrix}\right] *^{-1} D_2 * \left[\begin{matrix} 1 & 0
\\ 0 & \sqrt{-1} \end{matrix}\right].\end{equation}

Let $C_{\gr}$ be a $\Z_2$-graded complex algebra generated by $\Gamma(M, T^*_{\R}M)$ with generator relations as follows.
$$C_{\gr}:=<u\in \Gamma(M, T^*_{\R}M)\mid uv+uv=2g(u,v), \forall u,v\in \Gamma(M, T^*_{\R}M)>.$$
Although $C_{\gr}$ is isomorphic to $B_{\gr}$ through the map
$$B_{\gr}\to C_{\gr}: \, u\mapsto \sqrt{-1}u, \, \forall  u\in \Gamma(M, T^*_{\R}M),$$
we consider a different representation on $\Omega(M)$. Let $c_L'$ be the representation of $C_{\gr}$ on $\Omega(M)$ given by
\begin{equation}\label{II}c_L'(\sqrt{-1}u)=\left[\begin{matrix} 1 & 0
\\ 0 & -\sqrt{-1} \end{matrix}\right] *^{-1}  c_L(u) * \left[\begin{matrix} 1 & 0
\\ 0 & \sqrt{-1} \end{matrix}\right] , \quad  \forall u\in \Gamma(M, T^*_{\R}M).\end{equation}
Then
$$c_L'(u)(\omega)= (-1)^{k-1}  *^{-1}  c_L(u)  * \omega, \quad \forall u\in \Gamma(M, T^*_{\R}M), \forall \omega\in \Omega^k(M),$$
and therefore
$$c_L'(u)=(-1)^{\deg}c_R(u),\quad \forall u\in \Gamma(M, T^*_{\R}M).$$
From Kasparov \cite{Kasparov}, {\em the spectral triple representing Kasparov's fundamental class} is supposed to be
$$(\A_3, \mathcal{H}_3, D_3, \gamma_3)=(C_{\gr}, \Omega(M), d+d^*, (-1)^{\deg}),$$
which, because of (\ref{HH})(\ref{II}), is unitarily equivalent to $(\A_2, \mathcal{H}_2, D_2, \gamma_2)$. Thus we have the following theorem.
\begin{thm}The projective spectral triple
$(\A_1, \mathcal{H}_1, D_1, \gamma_1)$ is Morita equivalent to Kasparov's spectral triple $(\A_3, \mathcal{H}_3, D_3, \gamma_3)$.
\end{thm}

For a recent account of Kasparov's fundamental class on noncommutative Riemannian manifolds, we refer the reader to Lord-Rennie-Varilly \cite{Lord}.

\end{section} 
\begin{section}{Projective spectral triple as fundamental class in $K_0(M,W_3)$}
In this section we see how projective spectral triples represent the fundamental classes in the twisted K-homology $K^0(\A_{W_3})\cong K_0(M,W_3)$.

Denote by $\mathbf B_{\mathrm {gr}}$ the $\Z_2$-graded algebra of sections of Clifford bundle $\Cl(T^*M)$ over $M$ (even dimensional only), then every Clifford module $E$ over $M$ can be considered as a finitely generated projective right $\mathbf B_{\mathrm {gr}}^{\mathrm{op}}$-module, and a Clifford connection $\nabla^E$ on $E$ gives rise to a Dirac operator $D^E$ on $E$. Then $E\mapsto\ind\, D^E$ defines a canonical homomorphism
$$K_0(\mathbf B_{\mathrm {gr}}^{\mathrm{op}})\xrightarrow{\ind} \Z.$$
By Morita equivalence, $K_0(\mathbf B_{\mathrm {gr}}^{\mathrm{op}})$ can be replaced by the K-theory of an ungraded algebra, $K_0(\A_{W_3})$, and the homomorphism $\ind$ becomes the operation of pairing with the projective spectral triple.

\begin{thm}[Poincar\'e duality]\label{duality} For an even dimensional closed oriented manifold $M$, the projective spectral triple $\varsigma=(\A_{W_3},\mathcal H_{W_3},D_{W_3},\gamma_{W_3})$ represents the twisted K-orientation as a cycle of the twisted K-homology $K^0(\overline{\A_{W_3}})\cong K_0(M,W_3)$, and hence gives rise to the Poincar\'e duality
$$K^0(M,W_3-c)\xrightarrow[\cong]{\frown[\varsigma]}K_0(M,c), \quad\text{or}\quad K^{\bullet}(M,c)\times K^{\bullet}(M,W_3-c)\xrightarrow[\text{\rm pairing}]{ \text{\rm nondegenerate}} \Z,$$
for all $c\in H^3(M,\Z).$
Here the cap product can be defined by $[\E]\frown [\varsigma]=[\varsigma^{\E}]$ for any finite Kasparov $C^{\infty}(M)$-$\A_{W_3}$-module $\E$.

For odd dimensional $M$, the Poincar\'e duality can be formulated as
$$K^{\bullet}(M,c)\times K^{{\bullet}+1}(M,W_3-c)\xrightarrow[\text{\rm pairing}]{\text{\rm nondegenerate}} \Z, \quad\forall c\in H^3(M,\Z).$$
\end{thm}

See Kasparov \cite{Kasparov}, Carey-Wang \cite{CareyWang}, and Wang \cite{Bai-Ling} for details. When $c$ is $0$, this is a special case of the second Poincar\'e duality theorem \cite{Kasparov} in KK-theory.

\end{section} 
\begin{section}{Local index formula for projective spectral triples}\label{localformula}
In this section we present a local index formula associated to the projective spectral triple for every closed oriented Riemannian manifold $M$ of dimension $2n$. Let $\A=\smooth(M)$. Denote by
$$\varsigma=(\B,\mathcal H,D,\gamma)=(\A_{W_3},\mathcal H_{W_3},D_{W_3},\gamma_{W_3})$$
the projective spectral triple of $M$ defined in the preceding sections. Suppose a K-class $[p]$ or $[\E]$ in $K_0(\B)$ is represented by a projection matrix $p=(p_{ij})\in \M_m(\B)$ or by a right $\B$-module $\E=p \B^m$ respectively.
Let $D^{\E}$ denote the twisted Dirac operator on $\mathcal H^{\E}=\E\otimes_{\B}\mathcal H=p\mathcal H^m$ associated to the projective universal connection $\nabla^{\E}:\E\to\E\otimes_{\B}\Omega^1_u(\B)$ on $\E$, namely $\nabla^{\E}(p\boldsymbol b)=p \delta_u(p\boldsymbol b)$ and $D^{\E}(p\boldsymbol h)=pD(p\boldsymbol h)$, $\forall \boldsymbol b \in \B^m$, $\forall \boldsymbol h \in \mathcal H^m$.

The left $\B$-module $\mathcal H=\mathcal H_+\oplus\mathcal H_-$ is $\Z_2$-graded and so is $\mathcal H^{\E}=\mathcal H_+^{\E}\oplus\mathcal H_-^{\E}$. Denote by $D^{\E}_{\pm}$ the restrictions of $D^{\E}$ to $\mathcal H^{\E}_{\pm}\to \mathcal H^{\E}_{\mp}$. The index of $D^{\E}$ is
$$\ind(D^{\E})=\dim \ker D^{\E}_+ - \dim \ker D^{\E}_-.$$
Using the well-known local index formula (cf. \cite{GetzlerBook}), we have
$$\ind (D^{\E})=\int_M \hat{A}(M)\ch(\mathcal H^{\E}/\mathcal S).$$
$\hat{A}(M)$ is the $\hat{A}$-genus of the manifold $M$,
$$\hat{A}(M)=\det\nolimits^{1/2}\left({R/2\over \sinh(R/2)}\right)\in\Omega^{\ev}(M).$$
The relative Chern character $\ch(\mathcal H^{\E}/\mathcal S)$ is explained as follows. We consider $\mathcal H$ and $\mathcal H^{\E}$ as right Clifford modules with right Clifford actions $c_R$. The connection $\nabla:\mathcal H\to \mathcal H\otimes_{\A}\Omega^1(M)$ on $\mathcal H$ induced by the Levi-Civita connection on $M$ is a right Clifford connection. We can define a right Clifford connection $\nabla^{\mathcal H^{\E}}:\mathcal H^{\E}\to \mathcal H^{\E}\otimes_{\A}\Omega^1(M)$ on $\mathcal H^{\E}$ by $\nabla^{\mathcal H^{\E}}(p\boldsymbol h)=p\nabla(p\boldsymbol h)$. Denote by $R^{\mathcal H^{\E}}\in \End_{\A}(\mathcal H^{\E})\otimes_{\A}\Omega^2(M)$ the curvature of the connection $\nabla^{\mathcal H^{\E}}$,
$$R^{\mathcal H^{\E}}=\nabla^{\mathcal H^{\E}} \nabla^{\mathcal H^{\E}}=p(\nabla p)(\nabla p)+p\nabla^2\circ p,$$
and denote by $T$ the twisting curvature, that is $T=R^{\mathcal H^{\E}}-R^{\mathcal S}$, where
$$R^{\mathcal S}=c_R(R)={1\over 4} R_{ijkl}c_R(e_l)c_R(e_k)e_i\wedge e_j,$$
and $R_{ijkl}$ are the components of the Riemannian curvature tensor on $M$ under an orthonormal frame $\{e_i\}$. One can verify that $T=p(\nabla p)(\nabla p)-p c_L(R)p$. With the above notations, the relative Chern character is
$$\ch(\mathcal H^{\E}/\mathcal S)=2^{-n}\tr \exp(-T).$$
So we have an explicit local index formula
\begin{equation}\label{localindex}\ind(D^{\E})=2^{-n}\int_M \hat{A}(M)\tr \,\exp(-p(\nabla p)(\nabla p)+p c_L(R)p).
\end{equation}
From the viewpoint of noncommutative geometry,
$$\ind (D^{\E})=<[p],[\varsigma]>=<\ch[p],\ch[\varsigma]>,$$
where $\ch[p]\in HP_0(\B)$ and $\ch[\varsigma]\in HP^0(\B)$ are the periodic Connes-Chern characters of $[p]$ and $[\varsigma]$ respectively.
On the other hand, in terms of twisted Chern characters, as defined below,
\begin{equation}\label{relation}\ch_{W_3}[p]:=\Chkr ( \ch[p]) \in H^{\ev}(M,\C), \quad \ch_{W_3}[\varsigma]:=(\Chkr^*)^{-1}(\ch[\varsigma])\in H_{\ev}(M,\C),
\end{equation}
the index pairing can be written as
$$\ind (D^{\E})=<[p],[\varsigma]>=<\ch_{W_3}[p],\ch_{W_3}[\varsigma]>.$$
We now try to give local expressions of $\ch[p]$, $\ch[\varsigma]$, $\ch_{W_3}[p]$, and $\ch_{W_3}[\varsigma]$ as well as their relation (\ref{relation}) explicitly. The periodic Connes-Chern character $\ch[p]$ is represented by a sequence of cyclic cycles $\{\ch^{\lambda}_0(p),\ch^{\lambda}_2(p),...\}$, where
$$\ch^{\lambda}_{2m}(p)=(-1)^m{(2m)!\over m!}\tr(p^{\otimes 2m+1}) \in C^{\lambda}_{2m}(\B).$$
This sequence satisfies the periodicity condition $S(\ch^{\lambda}_{2m+2}(p))=\ch^{\lambda}_{2m}(p)$. An alternative way to represent $\ch[p]$ is to use normalized $(\bb,\BB)$-cycles, that is
$$\ch_{2m}^{(\bb,\BB)}(p)=(-1)^m{(2m)!\over m!}\tr((p-{1\over2})\otimes p^{\otimes 2m}).$$

As for the Connes-Chern character of $\varsigma$, one can apply the Connes-Moscovici \cite{ConnesMoscovici} local index formula to get a normalized $(\bb,\BB)$-cocycle. However, when trying to derive from Connes-Moscovici's formula an expression in terms of integrals of differential forms on $M$, one will be confronted with a very much involved calculation of Wodzicki residues of various pseudo-differential operators. On the other hand, based on the appearance of formula (\ref{localindex}), one can get a $C_{\lambda}$-cocycle $\ch_{\lambda}(\varsigma)=\sum\limits_m \ch_{\lambda}^{2m}(\varsigma)$ as follows:

Let $T(b_1,b_2)=(\nabla b_1)(\nabla b_2)-b_1c_L(R)b_2$, and define $\rho^0_{2m}:\B^{\otimes 2m+1}\to \Omega^{2m}(M)$ by
$$\rho^0_{2m}(b_0,...,b_{2m})= {1\over (2m)!}\tr (b_0 T(b_1,b_2)\cdots T(b_{2m-1},b_{2m})).$$
Then the relative Chern character $\ch(\mathcal H^{\E}/\mathcal S)=2^{-n}\rho^0_{2m}(\ch^{\lambda}_{2m}[p])$.

It is easily seen that $2^{-n}\int_M \hat A(M)\rho^0_{2m}(b_0,...,b_{2m})$ is a Hochschild cocycle but not cyclic cocycle if $m\ge 2$. By applying Theorem \ref{alternative}, we know that
\begin{equation}\tag{\ref{rho}}\rho_{2m}(b_0,...,b_{2m})={1\over (2m)!}\tr(b_0 \psi_{2m}(b_1,...,b_{2m}))\end{equation}
is a cyclic cocycle, and that $\rho_{2m}(\ch^{\lambda}_{2m}(p))=\rho^0_{2m}(\ch^{\lambda}_{2m}(p))$ for all $p$ with $[p]\in K_0(\B)$.
Thus by Theorem \ref{alternative} and the duality theorem (Thm. \ref{duality}), we have the following conclusions:
\begin{thm}\label{character}The cyclic cocycle $\ch_{\lambda}(\varsigma)=\sum\limits_m \ch_{\lambda}^{2m}(\varsigma)$, where
$$\ch_{\lambda}^{2m}(\varsigma)(b_0,...,b_{2m})=2^{-n}\int_M \hat A(M)\rho_{2m}(b_0,...,b_{2m}),\quad \forall b_i\in \B,$$
represents the Connes-Chern character $\ch[\varsigma]$ of the projective spectral triple $\varsigma$.
\end{thm}
\begin{thm}The Connes-Chern character and the twisted Chern character are related by
$$\ch[\varsigma]=\ch_{W_3}[\varsigma]\circ \sum_m \rho_{2m} \text{ and }\ \ch_{W_3}[p]=\sum_m \rho_{2m}(\ch^{\lambda}_{2m}[p])$$
as identical periodic cyclic cohomology classes and de Rham cohomology classes respectively.
\end{thm}
\begin{cor}
The twisted Chern characters of $[p]$ and $[\varsigma]$ can be represented by
$$\ch_{W_3}[p]= 2^n \,\ch(\mathcal H^{\E}/\mathcal S) \text{ and } \ \ch_{W_3}[\varsigma]=2^{-n}[\hat A(M)]\frown[M]$$
respectively.
\end{cor}

\end{section}

\end{document}